\documentclass[a4paper]{amsart}

\usepackage{graphicx}
\usepackage{amsmath}
\usepackage{amscd}
\usepackage{amsfonts}
\usepackage{mathrsfs}
\usepackage{mathtools}
\usepackage{graphicx}
\usepackage{hyperref}
\usepackage{here}
\usepackage{algorithmic}
\usepackage{algorithm}
\usepackage{url}
\usepackage{stmaryrd}
\usepackage[title]{appendix}
\usepackage{amssymb}
\usepackage[all]{xy}
\usepackage{color}
\usepackage{extarrows}
\usepackage{multirow}
\usepackage{enumitem}
\usepackage{comment}
\usepackage[normalem]{ulem}

\usepackage[foot]{amsaddr}

\numberwithin{equation}{subsection}

\newtheorem{theorem}{Theorem}[subsection]
\newtheorem{lemma}[theorem]{Lemma}
\newtheorem{proposition}[theorem]{Proposition}
\newtheorem{corollary}[theorem]{Corollary}
\theoremstyle{definition}
\newtheorem{definition}[theorem]{Definition}

\newtheorem{remark}[theorem]{Remark}

\theoremstyle{plain}
\newtheorem*{Maintheor}{Main Theorem}

\theoremstyle{plain}
\newtheorem{theor}{Theorem}

\theoremstyle{definition}
\newtheorem{method}{Method}

\begin{document}

\title[Explicit {arithmetic} on curves of genus 3 and applications]{Some explicit {arithmetic} on curves of genus three and their applications}

\author[T. Moriya]{Tomoki Moriya}
\author[M. Kudo]{Momonari Kudo}
%
\address[T. Moriya]{School of Computer Science, University of Birmingham, UK}
\address[M. Kudo]{Fukuoka Institute of Technology, Japan}
\email[T. Moriya]{t.moriya@bham.ac.uk}
\email[M. Kudo]{m-kudo@fit.ac.jp}

\maketitle
\begin{abstract}

A {\it Richelot isogeny} between Jacobian varieties is an isogeny whose kernel is included in the $2$-torsion subgroup of the domain. 
\textcolor{black}{A} Richelot isogeny whose codomain is the product of two or more principally \textcolor{black}{polarized} abelian varieties is called a {\it decomposed} Richelot isogeny.

In this paper, we develop some explicit \textcolor{black}{arithmetic} on curves of genus $3$, including algorithms to compute the codomain of a decomposed Richelot isogeny.
As solutions to compute the domain of a decomposed Richelot isogeny, explicit formulae of defining equations for Howe curves of genus $3$ are also given.
Using the formulae, we shall construct an algorithm with complexity $\tilde{O}(p^3)$ (resp.\ $\tilde{O}(p^4)$) to enumerate all hyperelliptic (resp.\ non-hyperelliptic) superspecial Howe curves of genus $3$.
\end{abstract}

\section{Introduction}\label{sec:intro}

Throughout, a curve means a non-singular projective variety of dimension one.
\textcolor{black}{The complexity is measured by the number of arithmetic operations in a field.}
Let $k$ be an algebraically closed field of characteristic $p \neq 2$, and $(A,D)$ a principally polarized abelian variety over $k$.
For a curve $C$ of genus $g$ over $k$, its automorphism group is denoted by $\mathrm{Aut}(C)$.
We also denote respectively by $J(C)$ and $\Theta_C$ the Jacobian variety of $C$ and the principal polarization on $J(C)$ given by $C$.
We say that $J(C)$ is {\it decomposable} (resp.\ {\it completely decomposable}) if there exists an isogeny $J(C) \to J(C_1) \times \cdots \times J(C_n)$ (resp.\ $E_1 \times \cdots \times E_g$) for some curves $C_1, \ldots , C_n$ (resp.\ $g$ elliptic curves $E_1, \ldots , E_g$), and in this case such an isogeny is called a decomposed (resp.\ completely \textcolor{black}{decomposed}) isogeny.
Studying the decomposition of Jacobian varieties is classically important, see e.g., \cite{Ekedahl1993ExemplesDC}, \cite{Kani1989IdempotentRA}, \cite{JenniferPaulhus2008}, \cite{PR17}.
In particular, a method by Kani-Rosen provided in their celebrated paper~\cite{Kani1989IdempotentRA} enables us to find some curves $C_i$ as above as quotients of $C$ by subgroups of $\mathrm{Aut}(C)$, when $\mathrm{Aut}(C)$ is non-trivial.

This paper focuses on the decomposition of Jacobian varieties by {\it Richelot isogenies}.
Here, an isogeny $\phi : J(C) \to A$ is called a {\it Richelot isogeny} if $\phi^{\ast}(\Theta_C) \approx 2 D$, where ``$\approx$'' denotes algebraic equivalence.
Richelot isogenies are generalizations of $2$-isogenies of elliptic curves, and play important roles to analyze the security of cryptosystems constructed \textcolor{black}{by means of isogeny graphs}, see e.g., \cite{CDS}, \cite{CS}.
For a hyperelliptic curve $C$ of genus $2$ (resp.\ $3$), Katsura-Takashima \cite{KT2020} (resp.\ Katsura~\cite{katsura2021decomposed}) showed that the existence of a decomposed Richelot isogeny from $J(C)$ is equivalent to that of an order-$2$ (long) automorphism in the reduced automorphism group $\overline{\mathrm{Aut}}(C)$, where $\overline{\mathrm{Aut}}(C) := \mathrm{Aut}(C) / \langle \iota_C \rangle$ with the hyperelliptic involution $\iota_C$ of $C$.

They showed also in \cite{Katsura-Takashima} that any hyperelliptic curve \textcolor{black}{$C$} of given genus with an \textcolor{black}{involution $\sigma \in \mathrm{Aut}(C)$ different from $\iota_C$ (such a $\sigma$ is called an extra involution of $C$)} has a decomposed Richelot isogeny $\! J(C) \! \to \! J(C/\langle \sigma \rangle) \! \times \! J(C /\langle \sigma \! \circ \! \iota_C \rangle)$, and that any (generalized) {\it Howe curve} has a decomposed Richelot isogeny, where a Howe curve is the normalization of a fiber product of hyperelliptic curves.

In this paper, we start with writing down Katsura-Takashima's construction~\cite{Katsura-Takashima} for the hyperelliptic case as an explicit algorithm, so as to symbolically compute defining equations for the quotient curves $C/\langle \sigma \rangle$ and $C/\langle \sigma \circ \iota_C \rangle$ as above.
The complexity of the algorithm is also determined.
More precisely, we prove the following:

\begin{theor}[Theorem \ref{thm:compute decom hyper} below]
Let $C$ be a hyperelliptic curve of genus $g \geq 2$ over $k$ with an extra involution $\sigma$.
Given $C$, there is an algorithm (explicitly, Algorithm \ref{alg:decom richelot}) to compute equations defining the quotient curves $C/\langle \sigma \rangle$ and $C/\langle \sigma \circ \iota_C \rangle$, which define the codomain of a decomposed Richelot isogeny described above.
\textcolor{black}{
Moreover, all the computations in the algorithm are done in the quadratic extension of a subfield $K$ of $k$, over which all the Weierstrass points of $C$ are defined.
Furthermore, the number of arithmetic operations (in the quadratic extension field) required for executing the algorithm is upper-bounded by a constant.}
\end{theor}

In the case of genus $3$, algorithms for completely decomposed Richelot isogenies are also provided:

\begin{theor}[Theorems \ref{thm:complete hyper} and \ref{thm:decom richelot non-hyper} below]
Let $C$ be a curve of genus $3$ over $k$, say, either of the following:
\begin{enumerate}
    \item A smooth plane quartic $F(X,Y,Z) = 0$, where $F$ is a ternary quartic form over a subfield $K$ of $k$.
    \item A hyperelliptic curve $y^2 = f(x)$ with a square-free octic $f$ over $k$, where the roots of $f$ belongs to a subfield $K$ of $k$.
\end{enumerate}
We suppose that $C$ has two order-$2$ automorphisms $\sigma$ and $\tau$ with $\sigma \neq \tau$ such that $\sigma  \tau = \tau  \sigma$. \textcolor{black}{We also suppose $k$ is an algebraic closure of a finite field if $C$ is a non-hyperelliptic curve.}
Given $C$, there exists an \textcolor{black}{algorithm} to compute defining equations for $C/\langle \sigma \rangle$, $C/\langle \tau \rangle$, and $C/\langle \sigma  \tau \rangle$ as genus-$1$ curves, which define the codomain of a completely decomposed Richelot isogeny $J(C) \to J(C/\langle \sigma \rangle) \times J(C /\langle \tau \rangle) \times J(C/\langle \sigma \tau \rangle)$.
\textcolor{black}{
Moreover, all the computations in the algorithm are done in a finite extension of $K$, where the extension degree does not depend on the characteristic of $k$.
Furthermore, the number of arithmetic operations (in the finite extension field) required for executing the algorithm is upper-bounded by a constant.}
\end{theor}

The inverse problem is also considered:
Given $C_1, \ldots , C_n$, compute a defining equation for $C$ such that $J(C)$ is Richelot isogenous to $J(C_1) \times \cdots \times J(C_n)$.
In Section \ref{sec:main2}, we shall give a complete solution for this problem in the case of genus $3$.
In this case, the codomain of a decomposed Richelot isogeny is either $E \times J(H)$ for a genus-$1$ curve $E$ and a genus-$2$ curve $H$ or $E_1 \times E_2 \times E_3$ for genus-$1$ curves $E_1$, $E_2$, and $E_3$.
When $E$ and $H$ share four Weierstrass points or each two elliptic curves $E_i$ and $E_j$ share two Weierstrass points, it is straightforward that we may take $C$ as the normalization of $E \times_{\mathbb{P}^1} H$ or $E_1 \times_{\mathbb{P}^1} E_2$, which is nothing but a Howe curve of genus $3$, and the problem has been reduced into finding a defining equation for a Howe curve of genus $3$.
The former case is solved by a general fact on the fiber product of hyperelliptic curves, see Lemma \ref{lem:C/sigma hyper} below.
The following \textcolor{black}{main} theorem is our solution for the latter case:

\begin{Maintheor}[\textcolor{black}{Theorems \ref{thm:Howe curve hyper elliptic} and \ref{thm:Howe curve non-hyper2}}]
\textcolor{black}{Let $E_1\colon y^2=x(x-\alpha_2)(x-\alpha_3)$ and $E_2\colon y^2=x(x-\beta_2)(x-\beta_3)$ be elliptic curves, where $\alpha_2$, $\alpha_3$, $\beta_2$, and $\beta_3$ are pairwise distinct elements in $k \smallsetminus \{ 0 \}$.}
An equation defining the normalization $C$ of $E_1 \times_{\mathbb{P}^1} E_2$ is given as \eqref{eq:Howe hyper equation} (resp.\ \eqref{eq:Howe non-hyper equation}) in the hyperelliptic (resp.\ non-hyperelliptic) case. 
\end{Maintheor}

As an important corollary, \textcolor{black}{if we take $\alpha_2 = 1$, $\alpha_3 = \nu$, $\beta_2 = \mu$, and $\beta_3 = \mu \lambda$ with $\nu = \mu^2 \lambda$,} any coefficient of \eqref{eq:Howe hyper equation_Legendre} for superspecial $C$ belongs to $\mathbb{F}_{p^2}$, see Corollary \ref{cor:max} below (cf.\ a resent result~\cite[Theorem 1.1]{Ohashi1} by Ohashi).

Furthermore, computational methods to enumerate superspecial Howe curves of genus $3$ will be given in Section \ref{sec:ssp}, and their complexities are also determined:

\begin{theor}\label{thm:main4}
\textcolor{black}{When $k = \overline{\mathbb{F}_{p}}$ for a prime $p$}, all isomorphism classes of superspecial hyperelliptic (resp.\ non-hyperelliptic) Howe curves of genus $3$ are enumerated in $\tilde{O}(p^3)$ (resp.\ $\tilde{O}(p^4)$), where Soft-O notation omits logarithmic factors.
\end{theor}

In our methods, we apply \textcolor{black}{Main Theorem} to computing equations defining superspecial Howe curves, and use some criteria given in Subsection \ref{subsec:isom} to make the isomorphism classification efficient.

We implemented the enumeration methods on Magma~\cite{Magma}, and the source \textcolor{black}{code} are available at the web page of the second author, see Subsection \ref{subsec:imp} for the url.
Computational results obtained by our implementation will be described in Subsection \ref{subsec:imp}.

\section{Preliminaries}\label{sec:pre}
\textcolor{black}{This section is mainly devoted to a review of some basic facts on hyperelliptic curves, generalized Howe curves, and Richelot isogenies between abelian varieties.}

\subsection{Hyperelliptic curves and their isomorphisms.}\label{subsec:hyp}

\textcolor{black}{
Let $k$ be a field of characteristic $p$ with $p \neq 2$.
Let $C$ be a hyperelliptic curve of genus $g \geq 2$ over $k$, and $\pi_C : C \to \mathbb{P}^1$ its structure map of degree $2$.
Then} $C$ has a unique involution $\iota_C : C \to C$ such that $C/\langle \iota_C \rangle \cong \mathbb{P}^1$, which is called the {\it hyperelliptic involution} of $C$. 
Recall that $\iota_C$ is the unique non-trivial automorphism of $C$ \textcolor{black}{with} $\pi_C \circ \iota_C = \pi_C$.
We also call an involution of $C$ different from $\iota_C$ an {\it extra involution}.
A {\it Weierstrass point} of $C$ is an invariant point under $\iota_C$, and there are precisely $2g+2$ Weierstrass points of $C$.
It is well-known that any isomorphism of hyperelliptic curves maps Weierstrass points of the domain to ones of the codomain.

{
In general, a hyperelliptic curve $C$ can be represented by gluing the following two curves in $\mathbb{A}^2$:
\[
C_1\colon y^2=f_1(x),\quad C_2\colon y^2=f_2(x)
\]
via the map $C_1\smallsetminus\{(0,*)\}\to C_2\smallsetminus\{(0,*)\} \ ; \ (x,y)\mapsto (1/x,y/x^{g+1})$, where $f_1$ and $f_2$ are polynomials of degree $2g+1$ or $2g+2$ satisfying $f_1(1/x)x^{2g+2}=f_2(x)$.
Note that the map $C \rightarrow \mathbb{P}^1$ given by the projection onto the $x$-coordinate is a morphism of degree $2$.
For simplicity, we usually represent $C$ by $y^2=f_1(x)$ \textcolor{black}{or $y^2 = f_2(x)$}, and in this case the hyperelliptic involution $\iota_C$ is $(x,y) \mapsto (x,-y)$.
}

The following important lemma enables us to describe isomorphisms of hyperelliptic curves explicitly:

\begin{lemma}[{\cite[\S 2E1]{lercier2013fast}}]\label{lem:hyperelliptic matrix}
Let $C_1\colon y^2=f_1(x)$ and $C_2\colon y^2=f_2(x)$ be hyperelliptic curves of genus $g$ over $k$. Any isomorphism $\sigma \colon C_1 \to C_2$ can be represented by a pair $(A, \lambda) \in \mathrm{GL}_2(k)\times k^\times$ with $A = \alpha E_{11} + \beta E_{12} + \gamma E_{21} + \delta E_{22}$ for $\alpha,\beta, \gamma, \delta \in k$ such that
\begin{equation}\label{eq:isom}
\sigma(x,y)=\left(\frac{\alpha x+ \beta}{\gamma x+ \delta},\frac{\lambda y}{(\gamma x+ \delta)^{g+1}}\right),
\end{equation}
where $E_{ij}$ denotes the $2 \times 2$-matrix with $1$ in the $(i,j)$ entry and $0$'s elsewhere.
The representation is unique up to the equivalence $(A,\lambda)\equiv (\mu A,\mu^{g+1}\lambda)$ for $\mu\in k^\times$.
\end{lemma}

\textcolor{black}{
When $k = \overline{k}$ in Lemma \ref{lem:hyperelliptic matrix}, one can always take $\lambda = 1$, and $A$ is unique up to the multiplication by $(g+1)$-th roots of unity.
A generic method deduced from \textcolor{black}{Lemma \ref{lem:hyperelliptic matrix}} to decide whether two hyperelliptic curves $y^2=f_1(x)$ and $y^2=f_2(x)$ \textcolor{black}{over an algebraically closed field} are isomorphic is to compute a Gr\"{o}bner basis of a (projective zero-dimensional) homogeneous multivariate system derived from coefficients in $f_1(\alpha x + \beta z, \gamma x + \delta z) = f_2(x,z)$, see
e.g., \cite{KH18ful}.
Since the system has $4$ variables $\alpha$, $\beta$, $\gamma$, $\delta$ and since the maximal total degree $\leq 2g+2$, the complexity is estimated as $O(g^4)$ from \cite{ZeroDimGroebner}.
Alternative efficient methods in the case where $p$ does not divide $2g+2$ are proposed in \cite{lercier2013fast} with complexity $\tilde{O}(g)$, and they were implemented in Magma~\cite{Magma} as \texttt{IsIsomorphicHyperellipticCurves} (cf.\ \cite[\S 1.1]{LSR21}).}

\textcolor{black}{
Instead of the above methods, we can use {\it Igusa invariants}~\cite{Igusa60} (and absolute invariants) and {\it Shioda \it invariants}~\cite{ShiodaInvariants} in the genus-$2$ case and in the genus-$3$ case respectively.
{\it Dihedral invariants} are also useful when $\overline{\mathrm{Aut}}(C) \supset \mathbf{C}_2$ (equivalently, $\mathrm{Aut}(C) \supset V_4$), see \cite[Section 3]{GS05} for details.
{\it Homogeneous dihedral invariants} defined by Lercier–Ritzenthaler–Sijsling~\cite{lercier2016explicit} classify isomorphisms in the case where $\overline{\rm Aut}(C)$ is cyclic of order coprime to the characteristic $p$ of $k$.}

\subsection{Automorphisms of hyperelliptic curves.}\label{subsec:aut}

For a curve $C$ over a field $k$, we denote by $\mathrm{Aut}_k(C)$ the automorphism group of $C$ over $k$, and if $k = \overline{k}$, \textcolor{black}{we denote it simply by} $\mathrm{Aut}(C)$.
For a hyperelliptic curve $C : y^2 = f(x)$ over \textcolor{black}{an algebraically closed field}, the quotient group $\overline{\rm Aut}(C):=\mathrm{Aut}(C) / \langle \iota_C \rangle$ is called the {\it reduced automorphism group} of $C$.
\textcolor{black}{As a particular case of Lemma \ref{lem:hyperelliptic matrix}, any $\sigma \in \mathrm{Aut}_k(C)$ is represented as in \eqref{eq:isom} for some $(A,\lambda) \in \mathrm{GL}_2(k) \times k^{\times}$.}

\textcolor{black}{Throughout the rest of this paper, let $k$ be an algebraically closed field of characteristic $p \neq 2$.}
As noted in \cite[\S 2]{GS05} (without proof), any automorphism $\sigma$ of a hyperelliptic curve over \textcolor{black}{$k$} is represented by $\mathrm{diag}(\mu,1)\in \mathrm{GL}_2(k)$ for a primitive $\ell$-th root $\mu$ of unity with $\ell = \mathrm{ord}(\sigma)$, if $\ell$ is coprime to \textcolor{black}{$p$}.
In Proposition \ref{prop:aut of order ell} below, we give a complete proof for this fact in more \textcolor{black}{precise arguments}:

\begin{proposition}\label{prop:aut of order ell}
Let $C:y^2=f_0(x)$ be a hyperelliptic curve of genus $g$ over \textcolor{black}{an algebraically closed field $k$}, and $\ell$ a positive integer coprime to $\mathrm{char}(k)$.
Let $\sigma \in \mathrm{Aut}(C)$, and assume that $\sigma$ has order $\ell$ in the reduced automorphism group of $C$.
Then there exists an isomorphism $\rho \colon C' \to C$ for a hyperelliptic curve $C':y^2 = f(x)$ such that the automorphism $\rho^{-1}  \sigma \rho$ of $C'$ is represented by $(\mathrm{diag}(\mu,1),\mu')\in \mathrm{GL}_2(k) \times k^{\times}$, where $\mu$ is a primitive $\ell$-th root of unity, and where $\mu'$ is an element satisfying $(\mu')^{\ell} = 1$ or $-1$.
We also have $\mu' = \pm \mu^{g+1}$ if $\mathrm{deg}(f) = 2g+2$, and $\mu' =\pm \sqrt{\mu^{2g+1}}$ if $\mathrm{deg}(f) = 2g+1$. 
Hence, $\sigma$ is the hyperelliptic involution if and only if $(\mu,\mu')=(1,-1)$ and $\ell=1$.
\end{proposition}

\begin{proof}
Recall from Lemma \ref{lem:hyperelliptic matrix} that $\sigma$ is represented by some $(A,\lambda) \in \mathrm{GL}_2(k) \times k^{\times}$.
We may assume $\lambda = 1$ by $(A,\lambda) \equiv (\lambda^{-1/(g+1)}A, 1)$.
Since $A^{\ell} = a I_2$ and $\lambda^{\ell} = \pm a^{g+1}$ for some $a \in k^{\times}$, there exists a matrix $Q = \alpha E_{11} + \beta E_{12} + \gamma E_{21} + \delta E_{22} \in \mathrm{GL}_2(k)$ with $Q^{-1}AQ = B_0:=b^{-1}\mathrm{diag}(\mu_1, \mu_2)$ in $\mathrm{GL}_2(k)$, where $b \in k$ is an arbitrary element with $b^{\ell}=a^{-1}$, and where $\mu_1$ and $\mu_2$ are $\ell$th roots of $1$.
Letting $\rho$ be an isomorphism represented by $(Q, \nu)$ with an arbitrary $\nu \in k^{\times}$, we have the following commutative diagram:
\[
\begin{CD}
C @>{\sigma}>> C \\
@A{\rho}AA @VV{\rho^{-1}}V   \\
C' @>>> C' ,
\end{CD}
\]
where $C'$ is a hyperelliptic curve defined by $\left(\frac{\nu y}{(\gamma x + \delta)^{g+1} }\right)^2 = f_0\left( \frac{\alpha x + \beta}{\gamma x + \delta} \right)$ with an automorphism $\rho^{-1} \sigma \rho$ (of order $\ell$ in the reduced automorphism group $\overline{\mathrm{Aut}}(C')$) represented by $(B_0 , \lambda) $.
Putting $B := \mathrm{diag}(\mu,1)$ with $\mu:=\mu_1\mu_2^{-1}$ and $\mu' = (b \mu_2^{-1})^{g+1} \lambda$, we furthermore have $(B_0 , \lambda) \equiv ( (b \mu_2^{-1}) B_0, (b \mu_2^{-1})^{g+1} \lambda) = (B, \mu')$.
Note that $\mu$ is an $\ell$-th root of unity, and that $(\mu')^{\ell} = ((b \mu_2^{-1})^{g+1} \lambda)^{\ell} = (a^{-1})^{g+1} \lambda^{\ell} = \pm 1$.

We also show the assertion on $\mu'$.
For this, let $y^2 = f(x)$ be a hyperelliptic equation of $C'$, and put $F(X,Z) := Z^{2g+2}f(X/Z)$.
Since $\rho^{-1} \sigma \rho$ is an automorphism of $C'$, it follows from $F(\mu X, Z) = (\mu')^2 F(X,Z)$ that $f(\mu x) = (\mu')^2 f(x)$.
Comparing the $\mathrm{deg}(f)$-th coefficients in both sides, we have $\mu^{\mathrm{deg}(f)} = (\mu')^2$, as desired.
Furthermore $\mu$ is a primitive $\ell$-th root of unity since the order of $\sigma$ as an element of the reduced automorphism group of $C$ is $\ell$.
\end{proof}

\begin{remark}
In Proposition \ref{prop:aut of order ell}, the degree of $f$ is always equal to $2g+2$ if $\sigma$ is an involution (i.e., order $\ell=2$) as an element of the (not reduced) automorphism group $\mathrm{Aut}(C)$.
Indeed, we have $\mu = -1$ and $(\mu')^{\ell}=1$ in this case, and it follows from $\mu^{\deg (f)} = (\mu')^2$ that $\mathrm{deg}(f)$ must be even.
\end{remark}

For any subgroup $G$ of $\mathrm{Aut}(C)$, we can consider the {\it quotient curve} $C/G$.
Finding an equation for $C/G$ is not so easy in general, but it is possible under some reasonable assumptions on $C$ and $G$ such as in the case of the following lemma:

\begin{lemma}\label{lem:C/sigma hyper}
\textcolor{black}{
Let $C$ be a hyperelliptic curve $y^2 = f(x^2)$ of genus $g$ over $k$, where $f(x)$ is a square-free polynomial in $k[x]$ of degree $g+1$ with $f(0) \neq 0$.
Let $\iota_C$ be the hyperelliptic involution $(x,y) \mapsto (x,-y)$, and let $\sigma_1 : (x,y) \mapsto (-x,y)$ be an involution on $C$.
Then the quotient curves $C/\langle \sigma_1 \rangle$ and $C / \langle \sigma_2 \rangle$ with $\sigma_2 := \sigma_1 \circ \iota_C$ are hyperelliptic curves given by $y^2 = f (x)$ and $y^2 = xf(x)$, respectively.}

\textcolor{black}{
Conversely, let $C_1$ and $C_2$ be hyperelliptic curves of genera $g_1$ and $g_1$ (resp.\ $g_1$ and $g_1 + 1$) over $k$ respectively defined by $y^2 = f (x)$ and $y^2 = xf(x)$, where $f(x)$ is a square free polynomial of degree $2 g_1 + 1$ (resp.\ $2 g_1 + 2$) with $f(0) \neq 0$.
Then the normalization of the fiber product $C_1 \times_{\mathbb{P}^1} C_2$ is isomorphic to the hyperelliptic curve $C : y^2 = f(x^2)$ of genus $2 g_1$ (resp.\ $2g_1 + 1$).}
\end{lemma}

\subsection{Generalized Howe curves}\label{subsec:Howe}

In this subsection, we recall the notion of {\it generalized Howe curves}.
Consider two hyperelliptic curves $C_1$ and $C_2$ of genera $g_1$ and $g_2$ over an algebraically closed field $k$ with the following affine models:
\begin{equation}\label{eq:Howe}
\begin{aligned}
    C_1&: \textcolor{black}{y_1^2 = f_1(x) } = (x-a_1)\cdots(x-a_r)(x-b_1) \cdots (x-b_{2g_1+2-r}),\\
    C_2&: \textcolor{black}{y_2^2 = f_2(x)} = (x-a_1)\cdots(x-a_r)(x-c_1) \cdots (x-c_{2g_2+2-r}),
\end{aligned}
\end{equation}
where $a_i$, $b_i$, and $c_i$ are all distinct elements in $k$.
Assume that $g_1 \leq g_2$, and let $\pi_i: C_i \to {\mathbb P}^1$ be the usual double covers.
\textcolor{black}{As in \cite{Katsura-Takashima}, assume that there is no isomorphism $\varphi : C_1 \to C_2$ with $\pi_2 \circ \varphi = \pi_1$.}
As it is defined in \cite{Katsura-Takashima}, the normalization $C$ of the fiber product $C_1 \times_{{\mathbb P}^1\!} C_2$ is called a {\it generalized Howe curve}.
\textcolor{black}{Let $\sigma_i$ be an involution on $C$ whose quotient map is $C \to C_i$ for each $i=1,2$, and let $C_3$ be the quotient curve of $C$ by the other involution $\sigma_3:=\sigma_1\sigma_2$ on $C$.}
\textcolor{black}{We can write an equation for $C_1 \times_{\mathbb{P}^1}C_2$ locally as $y_1^2 - f_1(x) = y_2^2 - f_2(x) =0$, from which we obtain an equation for $C_3$ as follows:}
\[
    \textcolor{black}{C_3: y_3^2 = f_3(x) = (x-b_1) \cdots (x-b_{2g_1+2-r})(x-c_1) \cdots (x-c_{2g_2+2-r}),}
\]
\textcolor{black}{where we set $y_3 = \frac{y_1 y_2}{(x-a_1) \cdots (x-a_r)}$ as in \cite{Katsura-Takashima}.
We then have the following commutative diagram:}
\[
\textcolor{black}{
\xymatrix{
C \ar[r] \ar[d] & C_i \ar[d]^{\pi_i}\\
\mathbb{P}^1 \ar[r] & \mathbb{P}^1.
}
}
\]

We here recall the following proposition on the genus $g(C)$ of the generalized Howe curve $C$ and a criterion for $C$ to be hyperelliptic:

\begin{proposition}[{\cite[Proposition 1 and Theorem 2]{Katsura-Takashima}}]\label{Katsura-Takashima}
\textcolor{black}{Let $C_1$, $C_2$, and $C$} be as above.
Then we have $g(C) = 2g_1 + 2g_2 + 1-r$ and $r \leq g_1 + g_2 +1 $, and hence $g(C) \geq g_1 + g_2$.
Moreover, if $g(C) \geq 4$, then $C$ is hyperelliptic if and only if $r=g_1+g_2+1$, i.e., $C_3$ is rational.
\end{proposition}

As a particular case, the generalized Howe curve $C$ with $g_1=g_2=1$ and $r=1$ is a non-hyperelliptic curve of genus $g(C)=4$, and it is called simply a {\it Howe curve}, which was originally defined in \cite{Howe} (see also \cite{KHH} and \cite{KHS}).
On the other hand, in the case where $g(C) = 3$ with $g_1 = g_2 = 1$ and $r=2$, all the quotient curves $C_1$, $C_2$, and $C_3$ are of genus one, and this case was studied in \cite{Bouw}, \cite{Ohashi2}, \cite{Ohashi1}, and \cite{oort1991hyperelliptic}.
In particular, Oort~\cite{oort1991hyperelliptic} used this construction to prove the existence of superspecial curves of genus $3$, which we will introduce in Section \ref{sec:ssp} \textcolor{black}{below}.
As we will also recall at the beginning of Section \ref{sec:main2}, Oort also gave a criterion for $C$ of genus $3$ to be hyperelliptic (cf.\ Proposition \ref{Katsura-Takashima} for the case of $g \geq 4$).
Note that this $C$ of genus $3$ is called a {\it Ciani curve} if it is non-hyperelliptic, see \cite{Ciani}.

As it is described in \cite[\S 4]{Katsura-Takashima}, the hyperelliptic involutions of $C_1$ and $C_2$ lift to order-$2$ automorphisms on $C_1 \times_{\mathbb{P}^1} C_2$, and thus the automorphism group of $C$ contains a subgroup isomorphic to $V_4 := \mathbb{Z}/2\mathbb{Z} \times \mathbb{Z}/2\mathbb{Z}$.
By combining Proposition \ref{prop:aut of order ell} and Lemmas \ref{lem:C/sigma hyper}, we here prove the converse for the hyperelliptic case:

\begin{lemma}[cf.\ {\cite[Remark 2]{Katsura-Takashima}}]\label{lem:V4}
Let $C$ be a hyperelliptic curve of genus $g$ over $k$.
Then $C$ is a generalized Howe curve if and only if $\mathrm{Aut}(C)$ contains a subgroup isomorphic to $V_4$.
In this case, $C$ is birational to $C_1 \times_{\mathbb{P}^1} C_2$ for hyperelliptic curves $C_1$ and $C_2$ of genera $g_1$ and $g_2$ as in Lemma \ref{lem:C/sigma hyper}, where $g_1 = g_2$ (resp.\ $g_2 = g_1 + 1$) if $g$ is even (resp.\ odd).
\end{lemma}

\begin{proof}
The ``only if''-part was proved in \cite[\S 4]{Katsura-Takashima}, and thus we here prove the ``if''-part only.
Assume that $\mathrm{Aut}(C)$ contains a subgroup isomorphic to $V_4$.
In this case, $C$ has a non-hyperelliptic involution, and thus we may suppose by Proposition \ref{prop:aut of order ell} that $C$ is given by $ y^2 = f (x^2)$ for some square-free $f \in k[x]$ of degree $g+1$, with automorphisms $\sigma_1 :(x,y) \mapsto (-x,y)$ and $\sigma_2 : (x,y) \mapsto (-x,-y)$, where $\sigma_2 = \sigma_1 \circ \iota_{C}$.
We then obtain two hyperelliptic curves $C_1 := C / \langle \sigma_1 \rangle$ and $C_2:= C / \langle \sigma_2 \rangle$, which are defined respectively by $y^2 = f(x)$ and $y^2 = xf(x)$ from \textcolor{black}{the first half part of} Lemma \ref{lem:C/sigma hyper}.
By \textcolor{black}{the second half part of} Lemma \ref{lem:C/sigma hyper}, the fiber product $C_1 \times_{\mathbb{P}^1} C_2$ is birational to $C$.
\end{proof}

Note that an argument similar to Lemma \ref{lem:V4} is shown in \cite[Theorem 5]{JenniferPaulhus2008}, where the characteristic of $k$ is assumed to be $0$.

\subsection{\textcolor{black}{Decomposed Richelot isogenies.}}\label{subsec:Jacobian}

\textcolor{black}{Let $A$ be an abelian variety and $\hat{A}$ be its dual.
For an ample divisor $D$ on $A$, there is an isogeny $\Phi_D\colon A \to \textcolor{black}{\hat{A}}$ defined by $x \mapsto T^*_x(L)-L$, where $T_x\colon A \to A$ is a translation map $a \mapsto a+x$. 
If $\Phi_D$ is an isomorphism, we call $D$ a {\it principal polarization} on $A$, and a pair $(A,D)$ a {\it principally polarized abelian variety}.
We define a product of polarized abelian varieties as
\[
\prod_{i=1}^n(A_i,D_i)= \left(A_1\times \cdots \times A_n, \sum_{i=1}^n(A_1\times \cdots \times A_{i-1}\times D_i\times A_{i+1}\times \cdots \times A_n)\right).
\]
Let $C$ be a smooth projective curve over $k$ of genus $g\geq 1$\textcolor{black}{, and let $J(C)$ be a Jacobian variety of $C$.}
For a divisor $D$ on $C$ of degree $1$, there is a natural map $\alpha_C \colon C \hookrightarrow J(C)$ defined as $P \mapsto (P)-D$.
The theta divisor of $J(C)$ is defined as a sum of $g-1$ copies of $\alpha_C(C)$, say $\Theta_{C}:= \alpha_C(C) + \cdots +  \alpha_C(C)$.
It is known that $(J(C),\Theta_{C})$ is a principally polarized abelian variety.
If we have $J(C)\cong E^g$ for some supersingular elliptic curve $E$, we call $C$ a {\it supersupecial} (s.sp.\ for short) curve.
Recall from \cite[Theorem 3.5]{shioda1979supersingular} that for supersingular elliptic curves $E_1,\ldots, E_{2g}$ with $g \geq 2$, we have $E_1\times \cdots \times E_g \cong E_{g+1}\times \cdots \times E_{2g}$.}

Let $(A,D)$ be a principally polarized abelian variety over \textcolor{black}{an algebraically closed field $k$}, and let $C$ be a curve over $k$.
If an isogeny $\phi\colon J(C)\to A$ satisfies $\phi^*(\Theta_C)\approx 2D$, we call $\phi$ a \textit{Richelot isogeny}, where $\approx$ is an algebraic equivalence. We call $\phi$ a {\it decomposed Richelot isogeny} if there exists two \textcolor{black}{principally polarized} abelian varieties $(A_1,D_1)$ and $(A_2,D_2)$ such that $(A,D)\cong (A_1,D_1)\times (A_2,D_2)$.
If there are elliptic curves $(E_1,O_{E_1}),\ldots ,(E_g,O_{E_g})$ satisfying $(A,D)\cong \prod_{i=1}^g(E_i,O_{E_i})$, we call $\phi$ a {\it completely decomposed Richelot isogeny}.

In the following, we recall some results in \cite{katsura2021decomposed} and \cite{Katsura-Takashima} on decomposed Richelot isogenies.
Katsura-Takashima showed in \cite{Katsura-Takashima} that any generalized Howe curve always has a decomposed Richelot isogeny:

\begin{theorem}[{\cite[Theorem 3]{Katsura-Takashima}}]\label{thm:KTmain3}
Let $C$ be a generalized Howe curve associated with hyperelliptic curves $C_1$ and $C_2$, and let $\sigma$ and $\tau$ be the automorphisms on $C$ of order $2$ obtained by lifting the hyperelliptic involutions $\iota_{C_1}$ and $\iota_{C_2}$ \textcolor{black}{respectively}, say $C_1 \cong C/\langle \sigma \rangle$ and $C_2 \cong C/\langle \tau \rangle$.
We also let $C_3 := C / \langle \sigma \circ \tau \rangle$.
Then, there exists a decomposed Richelot isogeny $J(C) \to J(C_1) \times J(C_2) \times J(C_3)$.
\end{theorem}

Since a hyperelliptic curve with an extra involution is a generalized Howe curve by Lemma \ref{lem:V4}, it follows from Theorem \ref{thm:KTmain3} that we have the following corollary:

\begin{corollary}[{\cite[Theorem 1]{Katsura-Takashima}}]\label{cor:KTmain1}
Let $C$ be a hyperelliptic curve with an automorphism $\sigma$ of order $2$, which is not the hyperelliptic involution $\iota_C$.
Then, there exists a decomposed Richelot isogeny $J(C) \to J(C / \langle {\sigma} \rangle) \times J(C / \langle {\sigma} \circ \iota_C \rangle)$.
\end{corollary}

\begin{remark}
It follows from Lemma \ref{lem:V4} that we can take $C_1 := C / \langle {\sigma} \rangle$ and $C_2:=C / \langle {\sigma} \circ \iota_C \rangle$ in Corollary \ref{cor:KTmain1} to be hyperelliptic curves of genera $g_1$ and $g_2$ with $g_1 = g_2 = g/2$ (resp.\ $g_1 =(g-1)/2$ and $g_2 = (g+1)/2$) if $g$ is even (resp.\ odd).
Such a decomposition of $J(C)$ is already given in \cite[Theorem 5]{JenniferPaulhus2008} by Paulhus based on Kani-Rosen's decomposition method~\cite[Theorem B]{Kani1989IdempotentRA}, although she did not prove an isogeny providing the decomposition is a Richelot one. 
\end{remark}

In the case of genus $3$, Katsura~\cite{katsura2021decomposed} proved \textcolor{black}{stronger} arguments than Theorem \ref{thm:KTmain3} and Corollary \ref{cor:KTmain1}.
More precisely, he proved in \cite[\S 4]{katsura2021decomposed} that the converse of Corollary \ref{cor:KTmain1} also holds for genus-$3$ hyperelliptic curves:

\begin{theorem}\label{thm:decomposed Richelot}
Let $C$ be a hyperelliptic curve of genus $3$, and $\iota_C$ its hyperelliptic involution.
Then, $C$ has an extra involution $\sigma$ if and only if there exists a decomposed Richelot isogeny outgoing from $(J(C),\Theta_C)$.
In this case, the codomain of such a decomposed Richelot isogeny is given as $J(C/\langle\sigma \rangle)\times J(C/\langle \sigma\circ\iota_C \rangle)$.
\end{theorem}
\begin{proof}
See \cite[Theorem 1]{Katsura-Takashima} and Proposition \ref{prop:aut of order ell}.
\end{proof}

As for Theorem \ref{thm:KTmain3}, a (not necessarily hyperelliptic) Howe curve of genus $3$ associated with two genus-$1$ curves has two different automorphisms of order $2$.
These automorphisms are {\it long automorphisms} defined as follows:

\begin{definition}[Long automorphism, {\cite[Definition 3.4]{katsura2021decomposed}}]\label{def:long_aut}
Let $C$ be a curve of genus $g$, and $\sigma$ an order-$2$ automorphism of $C$.
We denote by $\sigma^{\ast}$ the action on $H^0(C,\Omega_C^1)$ induced from $\sigma$, where $\Omega_C^1$ is the sheaf of differential $1$-forms on $C$.
We call $\sigma$ a {\it long automorphism} if the $g$ eigenvalues of $\sigma^*$ are $1,-1,\ldots ,-1$ (the number of $-1$ is $g-1$).
\end{definition}

For an order-$2$ automorphism of a curve $C$, it is straightforward that $\sigma$ is a long automorphism if and only if the dimension of the space $\Gamma(C/\langle \sigma \rangle, \Omega_{C/\langle \sigma \rangle})$ is $1$, i.e., the quotient curve $C/\langle \sigma \rangle$ has genus $1$ (this also means that $C$ is {\it bielliptic}).
In the case of genus $3$, any order-$2$ automorphism of a non-hyperelliptic curve is a long automorphism (cf.\ \cite[Corollary 5.3]{katsura2021decomposed}), and moreover we have the following:

\begin{theorem}[{\cite[Proofs of Theorem 6.2 and Theorem 6.3]{katsura2021decomposed}}]\label{thm:completely decomposed Richelot}
Let $C$ be a curve of genus $3$ over $k$. 
Then the following are equivalent:
\begin{enumerate}
    \item $C$ is a Howe curve of genus $3$ associated with two genus-$1$ curves.
    \item $C$ has two long automorphisms $\sigma$ and $\tau$ with $\sigma\neq \tau$ such that $\sigma\circ\tau=\tau\circ\sigma$.
    \item[(3)] There is a completely decomposed Richelot isogeny $J(C) \to E_1 \times E_2 \times E_3$ for some genus-$1$ curves $E_1$, $E_2$, and $E_3$.
\end{enumerate}
In this case, with $\sigma$ and $\tau$ in (2), one can take $E_1$, $E_2$, and $E_3$ in (3) as $E_1:=C/\langle\sigma \rangle$, $E_2:=C/\langle\tau \rangle$, and $E_3:=C/\langle \sigma\circ\tau \rangle$ \textcolor{black}{respectively} such that $C$ is birational to $E_1 \times_{\mathbb{P}^1} E_2$.
\end{theorem}

\section{Computing decomposed Richelot isogenies}\label{sec:main1}

In this section, we present algorithms to explicitly compute the codomain of decomposed Richelot isogenies, based on the theory of \cite{katsura2021decomposed} and \cite{Katsura-Takashima} together with properties of automorphisms described in Subsection \ref{subsec:aut}.



\subsection{Case of hyperelliptic curves}\label{subsec:main1-hyp}
In this subsection, we explain a precise algorithm to compute decomposed Richelot isogenies of hyperelliptic curves. 
\begin{theorem}\label{thm:compute decom hyper}
Let $C$ be a hyperelliptic curve of genus $g$, and $\iota_C$ its hyperelliptic involution. Assume that $C$ has an extra involution $\sigma$.
For a given $C$, there is an algorithm (explicitly, Algorithm \ref{alg:decom richelot}) to compute defining equations for the quotient curves $C/\langle \sigma \rangle$ and $C/\langle \sigma \circ \iota_C \rangle$, which define the codomain of a decomposed Richelot isogeny in Theorem \ref{thm:decomposed Richelot}.
\textcolor{black}{
Moreover, all the computations in the algorithm are done in the quadratic extension of a subfield $K$ of $k$, over which all the Weierstrass points of $C$ are defined.
Furthermore, the number of arithmetic operations (in the quadratic extension field) required for executing the algorithm is upper-bounded by a constant.}
\end{theorem}

\begin{proof}
In the general case, it seems hard to compute $C/\langle \sigma \rangle$ and $C/\langle \sigma\circ\iota_C \rangle$ for some involution $\sigma$ from $C$. However, if $\sigma$ is the automorphism $(x,y)\mapsto (-x,y)$, then we can easily compute $C/\langle \sigma \rangle$ and $C/\langle \sigma\circ\iota_C \rangle$ from Lemma \ref{lem:C/sigma hyper}. Proposition \ref{prop:aut of order ell} claims that there is an isomorphism $\rho \colon C' \to C$ such that $\rho^{-1} \circ \sigma \circ \rho$ is the automorphism $(x,y)\mapsto (-x,y)$ or $(x,y)\mapsto (-x,-y)$ of $C'$.
By swapping $\sigma$ and $\sigma\circ \iota_C$ if necessary, it suffice to consider only the case that $\rho^{-1} \circ \sigma \circ \rho:(x,y)\mapsto (-x,y)$, and therefore constructing $\rho$ and $C'$ provides the target curves $C/\langle \sigma \rangle$ and $C/\langle \sigma\circ\iota_C \rangle$.

\textcolor{black}{Let $(x_1,0),\ldots ,(x_{2g+2},0)$ be the Weierstrass points of $C$. We have $(x_{t_1},0)=\sigma((x_1,0))$ and $(x_{t_2},0)=\sigma((x_s,0))$ 
for some $s,t_1,t_2\in\{1,\ldots,2g+2\}$ with $t_1 \neq 1$ and $s \neq 1, t_1$ and $t_2 \neq 1, t_1, s$. From \cite[Proposition A]{elkies1998shimura}, a representation matrix of $\sigma$ is given by
\[
\begin{pmatrix}
    x_sx_{t_2}-x_1x_{t_1} & (x_s+x_{t_2})x_1x_{t_1}-(x_1+x_{t_1})x_sx_{t_2} \\
    (x_s+x_{t_2})-(x_1+x_{t_1}) & x_1x_{t_1}-x_sx_{t_2}
\end{pmatrix}
.
\]
Note that if one of the Weierstrass points of $C$ is a point at infinity, the above matrix is constructed under a proper projectivization (i.e., use $(X_s:Z_s)$ with $X_s/Z_s=x_s$ for representing $(x_s,0)$).
Therefore, by diagonalizing the above matrix, we obtain a corresponding matrix of $\rho$ if $t_1,t_2,s$ are taken properly. That is, if $(x_{t_1},0)=\sigma((x_1,0))$ and $(x_{t_2},0)=\sigma((x_s,0))$, the automorphism $\rho^{-1}$ can be represented by
\[
\begin{pmatrix}
x_sx_{t_2}-x_1x_{t_1}+\lambda & (x_s+x_{t_2})x_1x_{t_1}-(x_1+x_{t_1})x_sx_{t_2} \\
-(x_sx_{t_2}-x_1x_{t_1})+\lambda & (x_1+x_{t_1})x_sx_{t_2}-(x_s+x_{t_2})x_1x_{t_1}
\end{pmatrix}
,
\]
where
{\small
\[
\lambda=\sqrt{(x_sx_{t_2}-x_1x_{t_1})^2+((x_s+x_{t_2})-(x_1+x_{t_1}))((x_s+x_{t_2})x_1x_{t_1}-(x_1+x_{t_1})x_sx_{t_2})}.
\]}
The candidates of $t_1$ are in the set $\{2,\ldots ,2g+2\}$, and those of $t_2$ are in the set $\{2,\ldots, 2g+2\}\smallsetminus \{s,t_1\}$. It is enough to take $s$ as the minimal number other than $1$ and $t_1$. Therefore, the number of the candidates of $\rho$ is at most $2(2g+1)(2g-1)=2(4g^2-1)$. 
If $\rho$ is a correct one, then the image points of Weierstrass points of $C$ under $\rho^{-1}$ can be represented by $(\pm x_1',0),\ldots, (\pm x_{g+1}',0)$ for some $x_1',\ldots ,x_{g+1}'\in k$. Conversely, if the image points of Weierstrass points of $C$ under $\rho^{-1}$ have forms of $(\pm x_1',0),\ldots, (\pm x_{g+1}',0)$, then $C'$ has an automorphism $(x,y)\mapsto (-x,y)$, and $\rho$ is the correct candidate.
Therefore, correct candidates of $\rho$ are provided by checking for each candidate whether the images of Weierstrass points under $\rho^{-1}$ are $(\pm x_1',0),\ldots, (\pm x_{g+1}',0)$ for some $x_1',\ldots ,x_{g+1}'\in k$.
}

\textcolor{black}{From} the above discussions, we can compute an isomorphism $\rho^{-1}$, and this gives an algorithm to compute the codomain of a decomposed Richelot isogeny outgoing from $J(C)$.
In Algorithms \ref{alg:find long auto hyper} and \ref{alg:decom richelot}, we write down a computational method described above in algorithmic format.

We finally determine the complexity of Algorithm \ref{alg:decom richelot}.
Let $K$ be a subfield of $k$ over which all the Weierstrass points of $C$ are defined, and $K'$ its quadratic extension field.
As for Algorithm \ref{alg:find long auto hyper} called at the line $1$, the number of iterations for each {\bf for}-loop is bounded by a constant, and clearly all the computations are done in a constant number of \textcolor{black}{arithmetic operations} on $K'$. \textcolor{black}{For} the same reason, the cardinality of $R$ is also bounded by a constant.
Clearly the lines $4$--$7$ require $O(1)$ \textcolor{black}{arithmetic operations on} $K'$, and consequently our assertion for the total complexity holds.
\end{proof}

\begin{algorithm}[t]
\textcolor{black}{
\caption{Compute extra involutions of hyperelliptic curves (\textsf{ExtraInv})} 
\label{alg:find long auto hyper} 
\begin{algorithmic}[1] 
\REQUIRE A hyperelliptic curve $C\colon y^2=(x-x_1)\cdots (x-x_{2g+2})$
\ENSURE A set of pairs of a matrix corresponding to an isomorphism $C\to C'$ that diagonalizes an extra involution of $C$ and all Weierstrass points of $C'$
\STATE $R \leftarrow \emptyset$
\FORALL{$t_1 \in \{2,\ldots,2g+2\}$}
\STATE $s\leftarrow$ the minimum integer in $\{2,\ldots,2g+2\}\smallsetminus \{t_1\}$
\FORALL{$t_2 \in \{2,\ldots,2g+2\}\smallsetminus \{t_1, s\}$}
\STATE $\alpha \leftarrow x_sx_{t_2}-x_1x_{t_1}$
\STATE $\beta \leftarrow (x_s+x_{t_2})x_1x_{t_1}-(x_1+x_{t_1})x_sx_{t_2}$
\STATE $\gamma \leftarrow (x_s+x_{t_2})-(x_1+x_{t_1})$
\STATE $\lambda \leftarrow \sqrt{\alpha^2+\beta\gamma}$
\STATE $\displaystyle x_1',\ldots ,x_{2g+2}' \leftarrow \frac{(\alpha+\lambda)x_1+\beta}{(-\alpha+\lambda)x_1-\beta},\ldots ,\frac{(\alpha+\lambda)x_{2g+2}+\beta}{(-\alpha+\lambda)x_{2g+2}-\beta}$
\IF{for all $i=1,\ldots , 2g+2$, there exists $j\neq i$ such that $x_i'=-x_j'$}
\STATE $R \leftarrow R\cup \left\{\left(\begin{pmatrix}\alpha+\lambda & \beta \\ -\alpha+\lambda & -\beta \end{pmatrix},(x_1',\ldots,x_{2g+2}')\right)\right\}$
\ENDIF
\ENDFOR
\ENDFOR
\RETURN $R$
\end{algorithmic}
}
\end{algorithm}

\begin{algorithm}[t]
\caption{Decomposed Richelot isogeny of hyperelliptic genus-$g$ curves} 
\label{alg:decom richelot} 
\begin{algorithmic}[1] 
\REQUIRE A hyperelliptic curve $C\colon y^2=(x-x_1)\cdots (x-x_{2g+2})$
\ENSURE A set of pairs of a curve $C_1$ of genus $\lfloor g/2 \rfloor$ and a curve $C_2$ of genus $\lceil g/2 \rceil$ such that there is a decomposed Richelot isogeny $J(C)\to J(C_1)\times J(C_2)$
\STATE $R \leftarrow \mathsf{ExtraInv}(C)$ (Algorithm \ref{alg:find long auto hyper})
\STATE $S\leftarrow \emptyset$
\FORALL{$(Q,(x_1',\ldots,x_{2g+2}'))\in R$}
\STATE Assume $x_{i}'=-x_{i+(g+1)}'$ for $i=1,\ldots, g+1$ by changing indexes
\STATE $f_1(x) \leftarrow (x-x_1'^2)\cdots (x-x_{g+1}'^2),\quad f_2(x) \leftarrow x\cdot f_1(x)$
\STATE $C_1\colon y^2=f_1(x),\quad C_2\colon y^2=f_2(x)$
\STATE $S \leftarrow S\cup \{(C_1,C_2)\}$
\ENDFOR
\RETURN $S$
\end{algorithmic}
\end{algorithm}

\subsection{Case of genus-3 hyperelliptic curves having completely decomposed Richelot isogenies}
Next, we explain a computational method for completely decomposed Richelot isogenies from Jacobian varieties of hyperelliptic curves of genus $3$.

\begin{theorem}\label{thm:complete hyper}
Let $C$ be a hyperelliptic curve of genus $3$. Assume that $C$ has two long automorphisms $\sigma$ and $\tau$ such that $\sigma \circ \tau =\tau \circ \sigma$.
For a given $C$, there is an algorithm to compute curves constructing the codomain of a completely decomposed Richelot isogeny in Theorem \ref{thm:completely decomposed Richelot}, that is, $C/\langle \sigma \rangle$, $C/\langle \tau \rangle$, and $C/\langle \sigma \circ \tau \rangle$. 
\textcolor{black}{
Moreover, all the computations in the algorithm are done in the quadratic extension of a subfield $K$ of $k$, over which all the Weierstrass points of $C$ are defined.
Furthermore, the number of arithmetic operations (in the quadratic extension field) required for executing the algorithm is upper-bounded by a constant.}
\end{theorem}
\begin{proof}
As shown in the proof of Theorem \ref{thm:compute decom hyper}, it is enough to find a hyperelliptic curve $C'$ and an isomorphism $\rho\colon C' \to C$ such that $(\rho^{-1}\circ \sigma \circ \rho )(x,y)=(-x,y)$ in order to compute a curve $C/\langle \sigma \rangle$. Therefore, to compute $C/\langle \sigma \rangle$ and $C/\langle \tau \rangle$, we need to find eigenvector matrices of two long automorphisms $\sigma$ and $\tau$ satisfying $\sigma\circ\tau =\tau\circ\sigma$.
\textcolor{black}{By using Algorithm \ref{alg:find long auto hyper}, we can find $\sigma$ and $\tau$ and corresponding eigenvectors of them. After checking the commutativity of $\sigma$ and $\tau$, we can compute $\sigma\circ\tau$ and $C/\langle\sigma\circ\tau\rangle$ from diagonalization. Proposition \ref{thm:completely decom hyper more eff} below provides the eigenvalues of the corresponding matrix of $\sigma\circ \tau$.}

\textcolor{black}{The assertions on the complexity follow immediately from Theorem \ref{thm:compute decom hyper}.}
\end{proof}

\textcolor{black}{For the eigenvalues of $\sigma\circ\tau$, we have the following proposition.}

\begin{proposition}\label{thm:completely decom hyper more eff}
Let $C$ be a hyperelliptic curve of genus $3$. Assume that $C$ has two long automorphisms $\sigma$ and $\tau$.
\textcolor{black}{Let $A$ and $B$ be matrices in $\mathrm{GL}_2(k)$ corresponding to $\sigma$ and $\tau$ respectively. Denote the eigenvalues of $A$ and $B$ by $\pm \lambda_A$ and $\pm \lambda_B$ respectively. Then, the eigenvalues of $AB$ are $\pm\sqrt{-1}\lambda_A\lambda_B$.}

\end{proposition}

\begin{proof}
Suppose that $\sigma\circ \tau=\tau\circ\sigma$. There is a constant value $c$ such that
\[
AB=c BA.
\]
Let $\alpha$ and $\beta$ be the eigenvalues of $AB$. Since $AB$ is a regular matrix, we have $\alpha \neq 0$ and $\beta \neq 0$.
Since $AB$ and $BA$ have same eigenvalues, it holds that
\[
(c \alpha,c\beta)=(\alpha,\beta) \text{ or }(c \alpha,c\beta)=(\beta,\alpha),
\]
and thus $c^2=1$.
We moreover claim $c=-1$.
Indeed, if $c =1$ (i.e., $AB=BA$), then $A$ and $B$ are simultaneously diagonalizable.
Thus it follows 
that $AB=\pm\textcolor{black}{\lambda_A\lambda_B}I_2$, which contradicts $\sigma \circ \tau \neq \mathrm{id}_C$.
Therefore we have $c=-1$ and $AB=-BA$. 
\textcolor{black}{Since $A^2=\lambda_A^2I_2$ and $B^2=\lambda_B^2I_2$, it holds that \[
(AB)^2=ABAB=-A^2B^2=-\lambda_A^2\lambda_B^2I_2.
\]
Therefore, the eigenvalues of $AB$ are $\pm\sqrt{-1}\lambda_A\lambda_B$.}
\end{proof}

\subsection{Case of smooth plane quartics (non-hyperelliptic curves of genus 3)}
In this subsection, we explain a way to compute decomposed Richelot isogenies outgoing from the Jacobian variety of a non-hyperelliptic curve of genus $3$. \textcolor{black}{In this subsection, we suppose that $k$ is an algebraic closure of $\mathbb{F}_p$ with $p \geq 5$.}
Note that from {\cite[Corollary 5.2]{katsura2021decomposed}}, such isogenies are always completely decomposed.


Let $C$ be a non-hyperelliptic curve of genus $3$ over $k$, which is isomorphic to a non-singular plane quartic $F(X,Y,Z)=0$, where $F$ is an irreducible quartic form in $X$, $Y$, and $Z$ over $k$.
Recall also that any isomorphism between plane non-singular curves over $k$ is represented by an element of $\mathrm{PGL}_3(k)$ uniquely.
From Theorem \ref{thm:completely decomposed Richelot}, if there is a completely decomposed Richelot isogeny outgoing from $J(C)$, then the automorphism group of $C$ contains $V_4$ as its subgroup, where $V_4$ is the dihedral group of order $4$, i.e., $V_{4} = (\mathbb{Z}/2\mathbb{Z})^2$.
Namely $C$ has different two long automorphisms $\sigma$ and $\tau$ with $\sigma\circ \tau=\tau\circ\sigma$.
Such non-hyperelliptic curves are isomorphic to curves of the following form (cf.\ \cite[Theorem 3.1]{lercier2014parametrizing}):
\begin{align}\label{eq:non-hyper special}
C'\colon a_1x^4+a_2y^4+a_3x^2y^2+a_4x^2+a_5y^2+a_6=0
\end{align}
with $a_i\in k$, which is referred to as standard form (cf.\ \cite{Bouw}, \cite{Ohashi2}).
In this case, we easily find long automorphisms of $C'$, say $\sigma  : (x,y)\mapsto (-x,y)$ and $\tau : (x,y)\mapsto (x,-y)$ with
$\sigma\circ \tau=\tau\circ \sigma$, and a completely decomposed Richelot isogeny
\[
J(C')\to C'/\langle \sigma\rangle \times C'/\langle \tau\rangle\times C'/\langle \sigma\circ \tau \rangle,
\]
where $C'/\langle \sigma\rangle$, $C'/\langle \tau\rangle$, and $C'/\langle \sigma\circ\tau\rangle$ are elliptic curves defined as:
\begin{align}\label{eq:non-hyper special target curves}
C'/\langle \sigma\rangle &\colon a_1x^2+a_2y^4+a_3xy^2+a_4x+a_5y^2+a_6=0,\nonumber\\
C'/\langle \tau\rangle &\colon a_1x^4+a_2y^2+a_3x^2y+a_4x^2+a_5y+a_6=0,\\
C'/\langle \sigma\circ\tau\rangle &\colon a_1x^4+a_2+a_3x^2+a_4x^2y+a_5y+a_6y^2=0.\nonumber
\end{align}
Therefore, our problem is reduced into finding a curve $C'$ in standard form \eqref{eq:non-hyper special} and an isomorphism between $C'$ and $C$ explicitly, in order to compute completely decomposed Richelot isogenies outgoing from $J(C)$.
For finding them, it suffices to compute a matrix in $\mathrm{PGL}_3(k)$ representing an isomorphism $\rho \colon C' \to C$ such that $(\rho^{-1}\circ\sigma\circ\rho)(X:Y:Z)=(-X:Y:Z)$ and $(\rho^{-1}\circ\tau\circ\rho)(X:Y:Z)=(X:-Y:Z)$.


\begin{theorem}\label{thm:decom richelot non-hyper}
With notation as above, the computation of a matrix in $\mathrm{PGL}_3(k)$ representing $\rho \colon C' \to C$ is reduced into that of solving a zero-dimensional multivariate system over a finite extension of a subfield $K$ of $k$, 
\textcolor{black}{where the extension degree does not depend on $p$, and where any coefficient of $F$ belongs to $K$.
Furthermore, the number of arithmetic operations (in the finite extension field) required for the computation is upper-bounded by a constant.}
\end{theorem}
\begin{proof}
In general, \textcolor{black}{any $A$ in $\mathrm{PGL}_3(k)$ of order $2$ has eigenvalues $1$ and $-1$}, so that it can be diagonalized into $\mathrm{diag}(-1,1,1)$
by multiplying $A$ by $-1$ if necessary.
Let $A_\sigma$ and $A_\tau$ be matrices corresponding to $\sigma$ and $\tau$, respectively, say, $A_\sigma A_\tau=A_\tau A_\sigma$.
Since $A_\sigma \neq A_\tau $, there is a matrix $Q$ such that
\[
A_\sigma=Q
\begin{pmatrix}
-1 & 0 & 0\\
0 & 1 & 0\\
0 & 0 & 1
\end{pmatrix}
Q^{-1}
\text{ and }
A_\tau=Q
\begin{pmatrix}
1 & 0 & 0\\
0 & -1 & 0\\
0 & 0 & 1
\end{pmatrix}
Q^{-1}.
\]
Then, the matrix $Q$ corresponds to the target isomorphism $\rho \colon C' \to C$,
and so it suffices to construct $Q$ for our purpose.
Note that if $A_{\sigma\tau}$ is a matrix in $\mathrm{PGL}_3(k)$ corresponding to $\sigma\circ\tau$, then it holds that
\[
A_{\sigma\tau}=Q
\begin{pmatrix}
1 & 0 & 0\\
0 & 1 & 0\\
0 & 0 & -1
\end{pmatrix}
Q^{-1}.
\]
Therefore, swapping columns of $Q$ is equivalent to swapping three long automorphisms $\sigma$, $\tau$, and $\sigma\circ\tau$. Moreover, we know that multiplying each column of $Q$ by a constant do not affect diagonalizations of $A_\sigma$, $A_\tau$, and $A_{\sigma\tau}$. Hence, it suffices to consider the cases that the matrix $Q$ is defined as follows:
\[
\begin{pmatrix}
a & b & c \\
d & e & f \\
1 & 1 & 1
\end{pmatrix}
\text{ or }
\begin{pmatrix}
a & b & c \\
1 & e & f \\
0 & 1 & 1
\end{pmatrix}
\text{ or }
\begin{pmatrix}
a & b & c \\
1 & 1 & f \\
0 & 0 & 1
\end{pmatrix}
\text{ or }
\begin{pmatrix}
1 & b & c \\
0 & e & f \\
0 & 1 & 1
\end{pmatrix}
\text{ or }
\begin{pmatrix}
1 & b & c \\
0 & 1 & f \\
0 & 0 & 1
\end{pmatrix}.
\]
Here, $C'$ is defined by $F((X:Y:Z)\cdot {}^t Q)=0$, and it follows from \eqref{eq:non-hyper special} that the coefficient of any monomial in $F((X,Y,Z)\cdot {}^t Q)$ other than $X^4$, $Y^4$, $X^2 Y^2$, $X^2 Z^2$, $Y^2 Z^2$, and $Z^4$ should be zero.
From this, for each of the above $5$ case, we obtain a multivariate system of quartic equations with respect to $a$, $b$, $c$, $d$, $e$, and $f$,
and \textcolor{black}{it follows from Lemma \ref{lem:root is bounded} that the number of its roots is bounded by a constant not depending on $p$.}

Since each of the systems consists of $9$ inhomogeneous polynomials of degree at most $4$ in at most $6$ variables, it follows from arguments in Remark \ref{rem:GB} below that computing the roots of each system with Gr\"{o}bner basis algorithms terminates in $\tilde{O}(1)$ \textcolor{black}{arithmetic operations} in a finite extension field $K'$ of $K$, where the extension degree does not depend on $p$.
It is clear that the other parts of the algorithm requires $O(1)$ \textcolor{black}{arithmetic operations} in $K'$.
\end{proof}

\begin{lemma}\label{lem:root is bounded}
The number of roots of each multivariate system constructed in Theorem \ref{thm:decom richelot non-hyper} is bounded by a constant not depending on $p$.
\end{lemma}
\begin{proof}
Since $\rho \circ [(x,y)\mapsto (-x,y)] \circ \rho^{-1}$ and $\rho \circ [(x,y)\mapsto (x,-y)] \circ \rho^{-1}$ are long automorphisms of $C$, each $Q$ determined by a root of the multivariate system simultaneously diagonalizes at least one pair of long automorphisms. 
Let $\sigma$ and $\tau$ be different long automorphisms of $C$ such that $Q$ simultaneously diagonalizes these long automorphisms. Since the dimension of each common eigenspace of $\sigma$ and $\tau$ is $1$, the other matrices determined by roots that simultaneously diagonaize $\sigma$ and $\tau$ are obtained by swapping columns of $Q$. Therefore, the number of candidates of $Q$ that simultaneously diagonalizes $\sigma$ and $\tau$ is at most $6$. From \cite[Theorem 3.1]{lercier2014parametrizing}, the number of long automorphsims of a non-hyperelliptic curve is bounded by a constant not depending on $p$. Hence, our claim \textcolor{black}{on} the number of roots holds.
\end{proof}

\begin{remark}\label{rem:GB}
\textcolor{black}{In general, for a zero-dimensional affine system of inhomogeneous polynomials $f_1, \ldots , f_m \in K[x_1,\ldots , x_n]$ with $n \leq m$ whose total degrees are $d_1 \geq \ldots \geq d_m$ respectively, the maximal total degree of elements in the reduced Gr\"{o}bner basis of $\mathcal{F}=\{ f_1, \ldots , f_m \}$ with respect to an arbitrary monomial order is upper-bounded by $d_1 \cdots d_n$, see \cite[Theorem 5.4]{Ritscher2012DegreeBA}.
    This implies that the extension degree of an extension field $L$ of $K$ where all the roots in $\overline{K}^n$ of $\mathcal{F}$ are defined is bounded by a function of $d_1, \ldots , d_n$.}

    \textcolor{black}{A typical way of computing the roots of a zero-dimensional affine system $\mathcal{F}$ is as follows:
    Compute the reduced Gr\"{o}bner basis $\mathcal{G}$ of the homogenization of $\mathcal{F}$ with respect to a graded monomial order by e.g., the $F_4$ algorithm~\cite{faugere1999new}.
    Then, convert $\mathcal{G}$ to a Gr\"{o}bner basis $\mathcal{G}'$ with respect to a lexicographic monomial order by the FGLM basis conversion~\cite{faugere1993efficient}.
    Substituting $1$ for an extra variable for homogenization and reducing it, we obtain the reduced Gr\"{o}bner basis $\mathcal{G}''$ of the input $\mathcal{F}$ with respect to an induced lexicographic monomial order.
    We can compute the roots of $\mathcal{G}''$ by repeating the univariate polynomial factorization and evaluations, by extending base fields if necessary.
    All of these computations are done over $L$ defined as above.
    It is well-known that the complexity of the $F_4$ algorithm is bounded by a function of the number $n$ of variables, the maximal total degree $d$ of the input $\mathcal{F}$, and the number $m$ of the input polynomials.
    The complexity of the FGLM algorithm is $O(n D^3)$ operations in $K$, where $D$ denotes the number of roots of $\mathcal{F}$ in $\overline{K}^n$ counted with multiplicity.
    It is also well-known that $D$ is bounded by $d_1 \cdots d_n$.
    }

    \textcolor{black}{If $K = \mathbb{F}_q$ and if $n$, $m$, $d_1,\ldots , d_m$ are fixed (not depending on the characteristic $p$ of $K$), we can take $L = \mathbb{F}_{q^r}$ for some constant $r$ not depending on $p$.
    In this case, the complexities of the $F_4$ and FGLM algorithms are both $O(1)$ arithmetic operations in $K$.
    All the roots in $L^n$ of the input system can be computed by repeating the univariate polynomial factorization and evaluations, which are done in $\tilde{O}(1)$ arithmetic operations in $L$.}
\end{remark}

\section{Explicit construction of Howe curves of genus 3}\label{sec:main2}

\textcolor{black}{Let $k$ be an algebraic closure of $\mathbb{F}_p$ with $p \geq 3$.}
In this section, we determine defining equations of generalized Howe curves of genus $3$ explicitly.
Recall that a generalized Howe curve $C$ is a curve birational to the fiber product of two hyperelliptic curves of genera $g_1$ and $g_2$, and it follows from Proposition \ref{Katsura-Takashima} that $C$ has genus $3$ if and only if $(g_1,g_2) = (1,1)$ or $(1,2)$, namely $C$ is birational to either of the following:
\begin{itemize}
\item {\bf (Oort type)} $E_1 \times_{\mathbb{P}^1} E_2$ for two genus-$1$ curves $E_1$ and $E_2$ which share exactly $2$ Weierstrass points, or
\item {\bf (Howe type)} $E \times_{\mathbb{P}^1} H$ for a genus-$1$ curve $E$ and a genus-$2$ curve $H$ which share exactly $4$ Weierstrass points. 
\end{itemize}
Note that any $C$ of Howe type is always hyperelliptic by Proposition \ref{Katsura-Takashima}, and we may assume from Lemma \ref{lem:V4} that $E$, $H$, and $C$ are given as
\begin{equation}\label{eq:HoweType}
    E \colon y^2=f(x),\quad H \colon y^2=xf(x),\quad C : y^2=f(x^2),
\end{equation}
where $f$ is a quartic polynomial over {$k$} with non-zero discriminant.

On the other hand, any $C$ of Oort type can be either hyperelliptic or non-hyperelliptic.
\textcolor{black}{
Transforming the $2$ common Weierstrass points to $0$ and $\infty$ by an element of $\mathrm{PGL}_2(k)$, we may assume that $E_1$ and $E_2$ are given as follows:
\begin{equation}\label{eq:Oort0}
\begin{split}
    E_1 & :  y_1^2 = \phi_1 (x) = x(x-\alpha_2)(x-\alpha_3),\\
    E_2 & :  y_2^2 = \phi_2 (x) = x(x-\beta_2)(x-\beta_3),
\end{split}
\end{equation}
where $\alpha_2$, $\alpha_3$, $\beta_2$, and $\beta_3$ are pairwise distinct elements in $k \smallsetminus \{ 0 \}$.
If necessary, we can replace the equation for $E_1$ by a Legendre form:
Considering the transformation $x \mapsto \frac{x}{\alpha_2}$, we can replace $E_1$ and $E_2$ by
\begin{equation}\label{eq:OortCurve}
E_{\nu}: y^2=x(x-1)(x-\nu)\quad \mbox{and} \quad E_{\mu,\lambda}:y^2=x(x-\mu)(x-\mu\lambda)
\end{equation}
respectively, where $\nu = \frac{\alpha_3}{\alpha_2}$, $\mu = \frac{\beta_2}{\alpha_2}$, and $\lambda = \frac{\beta_3}{\beta_2}$.
In this case, recall from \cite[\S 5.11]{oort1991hyperelliptic} that $C$ is hyperelliptic if and only if $\nu=\mu^2\lambda$, which is equivalent to $\alpha_2 \alpha_3 = \beta_2 \beta_3$ by $\nu - \mu^2 \lambda = \frac{\alpha_2 \alpha_3 - \beta_2 \beta_3}{\alpha_2^2}$.
}

In Subsection \ref{subsec:OortHyp} (resp.\ \ref{subsec:OortNonHyp}) below, we shall find an equation for the normalization $C$ of Oort type in the hyperelliptic (resp.\ non-hyperelliptic) case.
We also study the field of definition, the $\mathbb{F}_{p^2}$-maximality and the $\mathbb{F}_{p^2}$-minimality in the case where $C$ is superspecial.

\begin{remark}
Howe, Lepr{\'e}vost, and Poonen proposed in \cite{howe2000large} a formula to compute a curve of genus $3$ whose Jacobian variety is Richelot isogenous to a given product of three elliptic curves. Our method is different from their method. In fact, their method is similar to V\'{e}lu's formulas, while our method directly constructs a fiber product of elliptic curves.
\end{remark}

\begin{remark}
As for the case of genus $2$, Katsura and Oort~\cite[Section 2]{katsura1987supersingular} studied genus-$2$ (hyperelliptic) curves with automorphism group containing a subgroup isomorphic to $V_4$; such a genus-$2$ curve is nothing but a genus-$2$ generalized Howe curve $C$ birational to $E_1 \times_{\mathbb{P}^1} E_2$ with elliptic curves $E_1$ and $E_2$ (cf.\ Subsection \ref{subsec:Howe}).
In particular, they counted superspecial $C$'s with $\overline{\mathrm{Aut}}(C) = \mathbb{Z}/2\mathbb{Z}$ by investigating the supersingularity of $E_1$ and $E_2$ as quotient curves of $C$.
\end{remark}

\subsection{Hyperelliptic case}\label{subsec:OortHyp}
The following theorem determines an equation for a hyperelliptic Howe curve of Oort type:

\begin{theorem}\label{thm:Howe curve hyper elliptic}
\textcolor{black}{Let $\alpha_2,\alpha_3,\beta_2,\beta_3\in k$ such that these are not $0$ and different from each other. Assume that $\alpha_2\alpha_3=\beta_2\beta_3$, and let $E_1$ and $E_2$ be elliptic curves defined as in \eqref{eq:HoweType}.}
Then, a hyperelliptic Howe curve $C$ of \textcolor{black}{genus $3$ birational} to $E_1\times_{\mathbb{P}^1}E_2$ can be represented by
\textcolor{black}{
\begin{align}\label{eq:Howe hyper equation}
\begin{aligned}
y^2=&\left(x^2-\frac{(\sqrt{\alpha_2}+\sqrt{\alpha_3}+\sqrt{\beta_2}+\sqrt{\beta_3})^2}{(\alpha_2+\alpha_3-\beta_2-\beta_3)}\right)
\left(x^2-\frac{(\sqrt{\alpha_2}+\sqrt{\alpha_3}-\sqrt{\beta_2}-\sqrt{\beta_3})^2}{(\alpha_2+\alpha_3-\beta_2-\beta_3)}\right)\\
&\left(x^2-\frac{(\sqrt{\alpha_2}-\sqrt{\alpha_3}+\sqrt{\beta_2}-\sqrt{\beta_3})^2}{(\alpha_2+\alpha_3-\beta_2-\beta_3)}\right)\left(x^2-\frac{(\sqrt{\alpha_2}-\sqrt{\alpha_3}-\sqrt{\beta_2}+\sqrt{\beta_3})^2}{(\alpha_2+\alpha_3-\beta_2-\beta_3)}\right).
\end{aligned}
\end{align}}
\end{theorem}
\begin{proof}
Note that $C$ is a hyperelliptic curve from \textcolor{black}{the same discussion in} \cite[\S 5.11]{oort1991hyperelliptic}. To find an equation defining $C$, it suffices to compute image points of its Weierstrass points under the double cover $C\to \mathbb{P}^1$ of degree $2$. Since the Weierstrass points of $C$ are fixed points of the hyperelliptic involution of $C$, constructing the hyperelliptic involution provides the Weierstrass points. From \cite[\S 5.11]{oort1991hyperelliptic}, there is an order-$2$ automorphism $s$ of $\mathbb{P}^1$ that corresponds to the hyperelliptic involution of $C$ and satisfies \textcolor{black}{$s(0:1)=(1:0)$, $s(\alpha_2:1)=(\alpha_3:1)$, and $s(\beta_2:1)=(\beta_3:1)$}. By a straightforward computation, we have\textcolor{black}{
\[
Q
\begin{pmatrix}
\alpha_2\\
1
\end{pmatrix}
=
\begin{pmatrix}
\alpha_2\alpha_3\\
\alpha_2
\end{pmatrix}
,\quad
Q
\begin{pmatrix}
\beta_2\\
1
\end{pmatrix}
=
\begin{pmatrix}
\alpha_2\alpha_3\\
\beta_2
\end{pmatrix}
=
\begin{pmatrix}
\beta_2\beta_3\\
\beta_2
\end{pmatrix}
,\quad 
Q^2
=\alpha_2\alpha_3 I_2
,
\]
where $Q=\begin{pmatrix}
0 & \alpha_2\alpha_3 \\
1 & 0
\end{pmatrix}$. Therefore, $Q$ corresponds to the automorphism $s$.} Next, we find the fixed points of $Q$ in $E_1\times_{\mathbb{P}^1}E_2$ under the lifted automorphism of $s$. By considering the diagonalization of $Q$, we have\textcolor{black}{
\[
\begin{pmatrix}
-\sqrt{\alpha_2\alpha_3} & \sqrt{\alpha_2\alpha_3}\\
1 & 1
\end{pmatrix}
^{-1}
\begin{pmatrix}
0 & \alpha_2\alpha_3 \\
1 & 0
\end{pmatrix}
\begin{pmatrix}
-\sqrt{\alpha_2\alpha_3} & \sqrt{\alpha_2\alpha_3}\\
1 & 1
\end{pmatrix}
=\sqrt{\alpha_2\alpha_3}
\begin{pmatrix}
-1 & 0 \\
0 & 1
\end{pmatrix}.
\]
}Hence, there is an automorphism $\rho$ of $\mathbb{P}^1$ such that $(\rho^{-1}\circ s \circ \rho) (x)=-x$. The automorphism $\rho$ leads two isomorphism $\rho_1 \colon E_1' \to E_1$ and $\rho_2 \colon E_2' \to E_2$. By a direct computation, it holds that\textcolor{black}{
\begin{align*}
E_1' &\colon y^2= (x^2-1)\left(x^2-\left(\frac{\sqrt{\alpha_2}-\sqrt{\alpha_3}}{\sqrt{\alpha_2}+\sqrt{\alpha_3}}\right)^2\right),\\
E_2' &\colon y^2= (x^2-1)\left(x^2-\left(\frac{\sqrt{\beta_2}-\sqrt{\beta_3}}{\sqrt{\beta_2}+\sqrt{\beta_3}}\right)^2\right).
\end{align*}
}More precisely, $E_1'$ and $E_2'$ are curves obtained by gluing two curves as the definition of hyperelliptic curves, respectively. Namely, the curve $E_1'$ is a curve represented by gluing\textcolor{black}{
\[
y^2= (x^2-1)\left(x^2-\left(\frac{\sqrt{\alpha_2}-\sqrt{\alpha_3}}{\sqrt{\alpha_2}+\sqrt{\alpha_3}}\right)^2\right)\text{ and }y^2= (1-x^2)\left(1-\left(\frac{\sqrt{\alpha_2}-\sqrt{\alpha_3}}{\sqrt{\alpha_2}+\sqrt{\alpha_3}}\right)^2 x^2\right)
\]
}by a map $(x,y)\mapsto (1/x,y/x^2)$.
The curve $E_1\times_{\mathbb{P}^1}E_2$ is isomorphic to\textcolor{black}{
\[
E_1'\times_{\mathbb{P}^1}E_2' \colon
\begin{cases}
y_1^2= (x^2-1)\left(x^2-\left(\frac{\sqrt{\alpha_2}-\sqrt{\alpha_3}}{\sqrt{\alpha_2}+\sqrt{\alpha_3}}\right)^2\right)\\
y_2^2= (x^2-1)\left(x^2-\left(\frac{\sqrt{\beta_2}-\sqrt{\beta_3}}{\sqrt{\beta_2}+\sqrt{\beta_3}}\right)^2\right)
\end{cases}.
\]
}Note that the automorphism $s'\colon (x,y_1,y_2)\mapsto (-x,y_1,y_2)$ of $E_1'\times_{\mathbb{P}^1}E_2'$ corresponds to the hyperelliptic involution of $C$.
The fixed points of $E_1'\times_{\mathbb{P}^1}E_2'$ under $s'\colon (x,y_1,y_2)\mapsto (-x,y_1,y_2)$ are points whose $x$-coordinates are zero and points at infinity. These points are image points of Weierstrass points of $C$ under a birational map between $E_1'\times_{\mathbb{P}^1}E_2'$ and $C$. Now, we construct a map $E_1'\times_{\mathbb{P}^1}E_2' \to \mathbb{P}^1$ of degree $2$. From the definition of $s'$, we have\textcolor{black}{
\[
(E_1'\times_{\mathbb{P}^1}E_2')/\langle s' \rangle \colon
\begin{cases}
Y_1^2= (X-Z)\left(X-\left(\frac{\sqrt{\alpha_2}-\sqrt{\alpha_3}}{\sqrt{\alpha_2}+\sqrt{\alpha_3}}\right)^2Z\right)\\
Y_2^2= (X-Z)\left(X-\left(\frac{\sqrt{\beta_2}-\sqrt{\beta_3}}{\sqrt{\beta_2}+\sqrt{\beta_3}}\right)^2Z\right)
\end{cases}.
\]
}Here, we represent $(E_1'\times_{\mathbb{P}^1}E_2')/\langle s' \rangle$ via projective coordinates $(X:Y_1:Y_2:Z)$.
Note that the natural surjective map $E_1'\times_{\mathbb{P}^1}E_2' \to (E_1'\times_{\mathbb{P}^1}E_2')/\langle s' \rangle$ is of degree $2$.
We construct an isomorphism between $(E_1'\times_{\mathbb{P}^1}E_2')/\langle s' \rangle$ and $\mathbb{P}^1$. Let \textcolor{black}{$\alpha =\frac{\sqrt{\alpha_2}-\sqrt{\alpha_3}}{\sqrt{\alpha_2}+\sqrt{\alpha_3}}$ and $\beta=\frac{\sqrt{\beta_2}-\sqrt{\beta_3}}{\sqrt{\beta_2}+\sqrt{\beta_3}}$}. We have an isomorphism as follows:
\[
\begin{array}{ccc}
    (E_1'\times_{\mathbb{P}^1}E_2')/\langle s' \rangle & \longrightarrow & \frac{\beta^2-1}{\beta^2-\alpha^2}S^2 - \frac{\alpha^2-1}{\beta^2-\alpha^2}T^2 = U^2 \\
    (X:Y_1:Y_2:Z) & \longmapsto & (Y_1:Y_2:X-Z) \\
    \left(\frac{\beta^2 S^2-\alpha^2 T^2}{\beta^2-\alpha^2}:SU:TU:\frac{S^2-T^2}{\beta^2-\alpha^2}\right) & \longmapsfrom & (S:T:U)
\end{array}
.
\]
It is well known that a quadratic curve $\frac{\beta^2-1}{\beta^2-\alpha^2}S^2 - \frac{\alpha^2-1}{\beta^2-\alpha^2}T^2 = U^2$ is isomorphic to $\mathbb{P}^1$ by a map
\[{\footnotesize
\begin{array}{ccc}
    \frac{\beta^2-1}{\beta^2-\alpha^2}S^2 - \frac{\alpha^2-1}{\beta^2-\alpha^2}T^2 = U^2 & \longrightarrow & \mathbb{P}^1 \\
    (S:T:U) & \longmapsto & \left(\sqrt{\frac{\beta^2-1}{\beta^2-\alpha^2}}S + \sqrt{\frac{\alpha^2-1}{\beta^2-\alpha^2}}T:U\right)\\
    \left(\sqrt{\frac{\beta^2-\alpha^2}{\beta^2-1}}(V^2+W^2):\sqrt{\frac{\beta^2-\alpha^2}{\alpha^2-1}}(V^2-W^2):2VW\right) & \longmapsfrom & (V:W)
\end{array}
.}
\]
By considering the composition \textcolor{black}{of the above two maps}, we obtain a degree-$2$ map $\hat{x}\colon E_1'\times_{\mathbb{P}^1}E_2'\to\mathbb{P}^1$. From a direct calculation, the image points of Weierstrass points of $C$ under $C \to E_1'\times_{\mathbb{P}^1}E_2' \to (E_1'\times_{\mathbb{P}^1}E_2')/\langle s' \rangle$ are $(0:\pm\alpha:\pm\beta:1)$ and $(1:\pm 1:\pm 1:0)$. Therefore, the image points of Weierstrass points under $\hat{x}$ are\textcolor{black}{
\begin{align*}
\pm \alpha\sqrt{\frac{\beta^2-1}{\beta^2-\alpha^2}} \pm \beta \sqrt{\frac{\alpha^2-1}{\beta^2-\alpha^2}}& =\pm  \frac{\sqrt{\alpha_2}-\sqrt{\alpha_3}}{\sqrt{\alpha_2+\alpha_3-\beta_2-\beta_3}} \pm \frac{\sqrt{\beta_2}-\sqrt{\beta_3}}{\sqrt{\alpha_2+\alpha_3-\beta_2-\beta_3}},\\
\pm \sqrt{\frac{\beta^2-1}{\beta^2-\alpha^2}} \pm \sqrt{\frac{\alpha^2-1}{\beta^2-\alpha^2}}& =\pm  \frac{\sqrt{\alpha_2}+\sqrt{\alpha_3}}{\sqrt{\alpha_2+\alpha_3-\beta_2-\beta_3}} \pm \frac{\sqrt{\beta_2}+\sqrt{\beta_3}}{\sqrt{\alpha_2+\alpha_3-\beta_2-\beta_3}}.
\end{align*}
}\textcolor{black}{We denote these eight values by $x_1,\ldots,x_8$.} Then, the \textcolor{black}{target} curve $C$ is defined by $y^2=(x-x_1)\cdots (x-x_8)$. This completes the proof of Theorem \ref{thm:Howe curve hyper elliptic}.
\end{proof}

\textcolor{black}{
As a particular case of Theorem \ref{thm:Howe curve hyper elliptic}, we obtain the equation \eqref{eq:Howe hyper equation} for $H$ as
\begin{align}\label{eq:Howe hyper equation_Legendre}
\begin{aligned}
y^2=&\left(x^2-\frac{(1+\sqrt{\mu})^2(1+\sqrt{\mu\lambda})^2}{(1-\mu\lambda)(1-\mu)}\right)
\left(x^2-\frac{(1-\sqrt{\mu})^2(1+\sqrt{\mu\lambda})^2}{(1-\mu\lambda)(1-\mu)}\right)\\
&\left(x^2-\frac{(1+\sqrt{\mu})^2(1-\sqrt{\mu\lambda})^2}{(1-\mu\lambda)(1-\mu)}\right)\left(x^2-\frac{(1-\sqrt{\mu})^2(1-\sqrt{\mu\lambda})^2}{(1-\mu\lambda)(1-\mu)}\right),
\end{aligned}
\end{align}
when $E_1$ and $E_2$ are given by \eqref{eq:OortCurve}.
With this equation, we obtain} a result similar to Ohashi's one in \cite{Ohashi1} on the $\mathbb{F}_{p^2}$-maximality and the $\mathbb{F}_{p^2}$-minimality of hyperelliptic Howe curves of genus $3$:

\begin{corollary}[cf.\ {\cite[Theorem 1.1]{Ohashi1}}]\label{cor:max}
Assume that $E_1$ and $E_2$ are given by \eqref{eq:OortCurve} with $\nu = \mu^2 \lambda$.
Let $C$ be a hyperelliptic Howe curve $C$ of genus $3$ given in Theorem \ref{thm:Howe curve hyper elliptic}; put the equation \eqref{eq:Howe hyper equation_Legendre} as $y^2= \prod_{i=1}^4(x^2-A_i)$. 
If $C$ is superspecial, then each $A_i$ belongs to $\mathbb{F}_{p^2}$, and moreover it is $\mathbb{F}_{p^2}$-maximal (resp.\ $\mathbb{F}_{p^2}$-minimal) if $p \equiv 3 \pmod{4}$ (resp.\ $p \equiv 1 \pmod{4}$).
\end{corollary}

\begin{proof}
Assume that $C$ is superspecial.
Then $E_1$ and $E_2$ are both supersingular since the $p$-kernel $J(C)[p]$ is isomorphic to $J(E_1)[p] \times J(E_2)[p] \times J(E_3)[p]$, where $E_3$ is a genus-$1$ curve defined by {$y^2=(x-1)(x-\nu)(x-\mu)(x-\mu\lambda)$}.
Since the Legendre form of $E_2$ is $y^2=x(x-1)(x-\lambda)$, it follows from \cite[Proposition 3.1]{AT02} that $\nu$ and $\lambda$ are the $4$-th powers of elements in $\mathbb{F}_{p^2}$.
By $\mu^2 = \nu / \lambda$, we may assume that $\mu$ is the square of an element in $\mathbb{F}_{p^2}$.
Therefore $\sqrt{\mu}$ and $\sqrt{\mu \lambda}$ in the equation \eqref{eq:Howe hyper equation} for $C$ can be taken as elements in $\mathbb{F}_{p^2}$, which shows the first argument.
Here, $E_1$, $E_2$, and $E_3$ are $\mathbb{F}_{p^2}$-isomorphic to $y^2 =x (x-1)(x-\nu)$, $y^2 = x (x-1)(x-\lambda)$, and $y^2 = x (x-1)(x-\lambda')$ with
\begin{equation}\label{eq:lambda_prime}
\lambda' = \frac{(\mu \lambda - 1)(\mu-\nu)}{(\mu \lambda - \nu)(\mu-1)},
\end{equation}
and it follows also from \cite[Proposition 2.2]{AT02} that they are both $\mathbb{F}_{p^2}$-maximal (resp.\ $\mathbb{F}_{p^2}$-minimal) if $p \equiv 3 \pmod{4}$ (resp.\ $p \equiv 1 \pmod{4}$).
Since $C$ is $\mathbb{F}_{p^2}$-maximal (resp.\ $\mathbb{F}_{p^2}$-minimal) if and only if $E_1$, $E_2$, and $E_3$ are both $\mathbb{F}_{p^2}$-maximal (resp.\ $\mathbb{F}_{p^2}$-minimal), we have proved the second argument.
\end{proof}

\textcolor{black}{From} Lemma \ref{lem:V4} together with Theorem \ref{thm:Howe curve hyper elliptic}, we can immediately prove the following corollary which asserts that any hyperelliptic genus-$3$ curve of Oort type is of Howe type:

\begin{corollary}\label{cor:OortTypeIsHoweType}
Let $C$ be a hyperelliptic Howe curve of genus $3$ associated with $E_1 \times_{\mathbb{P}^1} E_2$.
Then $C$ is birational to $E \times_{\mathbb{P}^1} H$ for some genus-$1$ and genus-$2$ curves $E$ and $H$.
If $E_1$ and $E_2$ are given as in \eqref{eq:OortCurve}, then $E$ and $H$ are given respectively by $y^2=f(x)$ and $y^2=xf(x)$, where $f(x)$ is the unique polynomial such that $f(x^2)$ is equal to the right hand side of \eqref{eq:Howe hyper equation}
\end{corollary}

\subsection{Non-hyperelliptic case}\label{subsec:OortNonHyp}

\textcolor{black}{Let $C$ be a genus-$3$ Howe curve of Oort type associated with genus-$1$ curves $E_1$ and $E_2$ given by \eqref{eq:Oort0}.
Recall that any non-hyperelliptic curve of genus $3$ is canonically embedded into $\mathbb{P}^2$ as a non-singular plane quartic.
The aim of this subsection is to find an explicit quartic equation defining a non-hyperelliptic $C$.
We shall deduce such a quartic by computing a resultant from equations for $E_1$ and $E_2$.
Note that this idea can be applied to higher genus cases, see Remark \ref{rem:general} below.}
\textcolor{black}{
We set
\[
\begin{aligned}
Q_1 &:= (x-\alpha_2)(x-\alpha_3) = x^2 - \tau_1 x + \tau_2,\\
Q_2 &:= (x-\beta_2)(x-\beta_3) = x^2 - \rho_1 x + \rho_2,
\end{aligned}
\]
where $\tau_1 := \alpha_2 + \alpha_3$, $\tau_2 := \alpha_2 \alpha_3$, $\rho_1 := \beta_2 + \beta_3$, and $\rho_2 := \beta_2\beta_3$.
Putting $Y := y_1/x$ and $Z := y_2/x$ and squaring both sides, we obtain $x Y^2 = Q_1$ and $x Z^2 = Q_2$.
Put
\[
\begin{aligned}
f_1 &:= Q_1 - x Y^2 = x^2 + (- \tau_1 - Y^2) x  + \tau_2,\\
f_2 &:= Q_2 - x Z^2 = x^2 + (- \rho_1 - Z^2) x + \rho_2,\\
f &:= \mathrm{Res}_x (f_1,f_2) \in k[Y,Z].
\end{aligned}
\]
A straightforward computation shows that
\begin{equation}\label{eq:Howe non-hyper equation}
\begin{aligned}
    f &= \rho_2 Y^4 + (-\tau_2 - \rho_2) Y^2 Z^2 + \tau_2 Z^4 \\
& + (2\tau_1 \rho_2 - \tau_2 \rho_1 - \rho_1 \rho_2) Y^2 + (- \tau_1 \tau_2 - \tau_1 \rho_2 + 2 \tau_2 \rho_1)Z^2 \\
& + \tau_1^2\rho_2 - \tau_1 \tau_2 \rho_1 - \tau_1 \rho_1 \rho_2 + \tau_2^2 + \tau_2 \rho_1^2 - 2 \tau_2 \rho_2 + \rho_2^2,
\end{aligned}
\end{equation}
which is a quartic both in $Y$ and $Z$ by $\rho_2, \tau_2 \neq 0$ from $\alpha_2,\alpha_3,\beta_2,\beta_3 \neq 0$.
}
\begin{lemma}\label{lem:bij}
\textcolor{black}{
    With notation as above, we denote by $b_{ij}$ the coefficient of $Y^i Z^j$ in $f$.
    Then we have the following:
    \begin{align*}
           & b_{40} = \rho_2 = \beta_2 \beta_3, \quad b_{22} = - \tau_2 - \rho_2= - (\alpha_2 \alpha_3 + \beta_2 \beta_3), \quad b_{04} = \tau_2 = \alpha_2 \alpha_3 ,  \\
   & b_{20} = 2\tau_1 \rho_2 - \tau_2 \rho_1 - \rho_1 \rho_2, \qquad
        b_{02} = -\tau_1 \tau_2 - \tau_1 \rho_2 + 2 \tau_2 \rho_1, \\
      &  b_{00} = \Pi_{i, j} (\alpha_i - \beta_j) = (\alpha_1 - \beta_1)(\alpha_2 - \beta_1)(\alpha_1 - \beta_2)(\alpha_2 - \beta_2),\\
      & b_{22}^2 - 4 b_{40} b_{04} = (\tau_2 - \rho_2)^2 = (\alpha_2 \alpha_3 - \beta_2 \beta_3)^2,\\
      & b_{20}^2 - 4 b_{40} b_{00} = (\rho_1^2 - 4 \rho_2)(\tau_2 - \rho_2)^2 =(\beta_2-\beta_3)^2 (\alpha_2 \alpha_3 - \beta_2 \beta_3)^2,\\
     & b_{02}^2 - 4 b_{04} b_{00} = (\tau_1^2 - 4 \tau_2)(\tau_2 - \rho_2)^2 =(\alpha_2-\alpha_3)^2 (\alpha_2 \alpha_3 - \beta_2 \beta_3)^2,\\
     & b_{20}^2 b_{04}^2 - b_{20} b_{22} b_{02} + b_{02}^2 b_{40} + b_{22}^2 b_{00} - 4 b_{40} b_{00} b_{04} = (\tau_2 - \rho_2)^4 =  (\alpha_2 \alpha_3 - \beta_2 \beta_3)^4.
    \end{align*}
    Therefore, $b_{40}$, $b_{04}$, and $b_{00}$ are always non-zero.
    Moreover, if $C$ is non-hyperelliptic (i.e., $\alpha_2 \alpha_3 \neq \beta_2 \beta_3$), then the last four values are all non-zero.}
\end{lemma}

\begin{proof}
    \textcolor{black}{
    This is proved by a straightforward computation.}
\end{proof}

\textcolor{black}{
Here, we assume that $C$ is non-hyperelliptic, i.e., $\alpha_2 \alpha_3 \neq \beta_2 \beta_3$.}

\begin{proposition}\label{prop:irred}
    \textcolor{black}{
    With notation as above, $f$ is absolutely irreducible.}
\end{proposition}

\begin{proof}
\textcolor{black}{
    Assume for a contradiction that $f$ is reducible.
    We denote by $b_{ij}$ the coefficient of $Y^i Z^j$ in $f$.
    If $f$ has a linear factor, then it follows from Lemma \ref{lem:linear} that we have the system of equations \eqref{eq:system1}, which contradicts to the last assertion of Lemma \ref{lem:bij}. 
    Therefore, $f$ is factored into $f = b_{40} Q_1 Q_2$ for some monic irreducible quadratic polynomials $Q_1$ and $Q_2$, whence either of the four cases in Lemma \ref{lem:Q} holds.
    Among the four cases, the cases (I), (II), and (III) are impossible by Lemma \ref{lem:bij} together with our assumption $\alpha_2 \alpha_3 \neq \beta_2 \beta_3$.    
    The only possible case is (IV).
    In this case, it follows from Lemma \ref{lem:bij} that we can take $\alpha_2 \alpha_3 - \beta_2 \beta_3$ and $(\beta_2-\beta_3)(\alpha_2 \alpha_3 - \beta_2 \beta_3)$ as square roots of $b_{22}^2 - 4 b_{40} b_{04}$ and $b_{20}^2 - 4 b_{40} b_{00}$ respectively.
    Here, we have
    \[
    \begin{aligned}
        2b_{40} b_{02} - b_{22} b_{20} + \sqrt{b_{22}^2 - 4 b_{40} b_{04}} \sqrt{b_{20}^2 - 4 b_{40} b_{00}} = -2 \beta_3 (\alpha_2 \alpha_3 - \beta_2 \beta_3)^2,\\ 
        2b_{40} b_{02} - b_{22} b_{20} - \sqrt{b_{22}^2 - 4 b_{40} b_{04}} \sqrt{b_{20}^2 - 4 b_{40} b_{00}} =-2 \beta_2 (\alpha_2 \alpha_3 - \beta_2 \beta_3)^2, 
    \end{aligned}
    \]
    which are both non-zero by $\alpha_2 \alpha_3 \neq \beta_2 \beta_3$.
    This contradicts the equation (IV) in Lemma \ref{lem:Q}.}
\end{proof}

\textcolor{black}{
Moreover, we have the following theorem:}

\begin{theorem}\label{thm:Howe curve non-hyper2}
\textcolor{black}{
With notation as above, the projective closure $\tilde{C}$ of $f=0$ in $\mathbb{P}^2$ is non-singular, and $C$ is isomorphic to $\tilde{C}$.}
\end{theorem}

\begin{proof}
\textcolor{black}{
Denoting by $b_{ij}$ the $Y^i Z^j$-coefficient of $f$, we have the following:
\[
\frac{\partial f}{\partial Y} = 2 Y ( 2 b_{40} Y^2 +  b_{22} Z^2 +  b_{20} ), \quad
\frac{\partial f}{\partial Z} = 2 Z ( 2 b_{04} Z^2 + b_{22} Y^2 + b_{02})
\]
If $\frac{\partial f}{\partial Y} = \frac{\partial f}{\partial Z} = 0$, then we have
\begin{align*}
    \begin{cases}
    Y =0 \quad\text{or}\quad Y^2=\frac{- b_{22} Z^2 - b_{20}}{2b_{40}},\\
    Z =0 \quad\text{or}\quad Z^2=\frac{ -b_{22} Y^2 - b_{02}}{2 b_{04}},
    \end{cases}
\end{align*}
whence there are four cases.
By a tedious computation for each of the four cases together with Lemma \ref{lem:bij}, we obtain
\[
f(Y,Z)=
\begin{cases}
b_{00} = \prod_{i,j} (\alpha_i - \beta_j )& (Y=0\text{ and }Z=0),\\
- \frac{b_{20}^2 - 4 b_{40} b_{00}}{4b_{40}} = - \frac{(\beta_2-\beta_3)^2(\alpha_2\alpha_3 - \beta_2\beta_3)^2}{4 \beta_2 \beta_3} & (Y^2=\frac{- b_{22} Z^2 - b_{20}}{2b_{40}}\text{ and }Z = 0),\\
- \frac{b_{02}^2 -4 b_{00} b_{04} }{4b_{04}} = -\frac{(\alpha_2-\alpha_3)^2(\alpha_2\alpha_3 - \beta_2\beta_3)^2}{4 \alpha_2 \alpha_3}& (Y=0\text{ and }Z^2=\frac{ -b_{22} Y^2 - b_{02}}{2 b_{04}}),\\
\frac{b_{20}^2 b_{04}^2 - b_{20} b_{22} b_{02} + b_{02}^2 b_{40} + b_{22}^2 b_{00} - 4 b_{40} b_{00} b_{04}}{b_{22}^2 - 4 b_{40} b_{04}} & (\text{otherwise}),
\end{cases}
\]
where the right hand side in the fourth case is equal to $(\alpha_2\alpha_3 - \beta_2 \beta_3)^2$.
Therefore, it follows from $\alpha_2 \alpha_3 \neq \beta_2 \beta_3$ that $f(Y,Z) \neq 0$ for each of the four cases, whence $f(Y,Z) = 0$ is non-singular.}
{We next discuss the singularity of the points at infinity of $\tilde{C}$. It is easy to see that there is a singular point at infinity if and only if the polynomial $b_{40}Y^4+b_{22}Y^2Z^2+b_{04}Z^4$ has a multiple root. Since $b_{40},b_{04}\neq 0$ and $b_{22}^2-4b_{40}b_{04}\neq 0$, it does not have any multiple root, whence $\tilde{C}$ is non-singular.}

\textcolor{black}{
Next, we shall prove that $E_1 \times_{\mathbb{P}^1}E_2$ is birational to $D:f(Y,Z) = 0$.
For any point $(x,y_1,y_2)$ on $E_1 \times_{\mathbb{P}^1} E_2$, it is straightforward that the corresponding point $(Y,Z)$ lies on $C$, whence we obtain a rational map:
\[
\Phi : E_1 \times_{\mathbb{P}^1} E_2 \dashrightarrow D \ ; \ (x,y_1,y_2) \mapsto \left( \frac{y_1}{x}, \frac{y_2}{x} \right),
\]
which is well-defined over all the points $(x,y_1,y_2)$ in $E_1 \times_{\mathbb{P}^1}E_2$ with $x \neq 0$.
We can also construct the inverse rational map as follows:
For each point $(Y,Z)$ on $D$, it follows from $\mathrm{Res}_x (f_1(x,Y,Z),f_2(x,Y,Z)) = 0$ that there exists a common root of the univariate polynomials $f_1(x,Y,Z)$ and $f_2(x,Y,Z)$.
Since $\tau_2 \neq \rho_2$ by our assumption that $C$ is non-hyperelliptic, we have $- \tau_1 - Y^2 \neq - \rho_1 - Z^2$, and such the root is uniquely given by
\[
x = \frac{\tau_2 - \rho_2}{(Y^2 + \tau_1) - (Z^2 + \rho_1)} \neq 0
\]
With this root $x$, we define a rational map
\[
\Psi : D \dashrightarrow E_1 \times_{\mathbb{P}^1} E_2 \ ; \ (Y,Z) \mapsto \left( x, x Y, x Z \right),
\]
which is defined over all the points on $D$.
Clearly both $\Phi\circ\Psi$ and $\Psi \circ \Phi$ are well-defined, and they are equal to $\mathrm{\rm id}_D$ and $\mathrm{id}_{E_1\times_{\mathbb{P}^1}E_2}$ as rational maps. 
{Since $C$ and $\tilde{C}$ are smooth,}
we have completed the proof.}
\end{proof}

\textcolor{black}{
As a particular case, we can take an elliptic curve in Legendre form as $E_1$, by setting $(\alpha_2,\alpha_3,\beta_2,\beta_3) = (1,\nu,\mu,\mu \lambda)$ with $\alpha_2 \alpha_3 = \nu \neq \mu^2 \lambda = \beta_2\beta_3$.
In this case, similarly to the hyperelliptic case, we obtain the following corollary (cf.\ \cite{Ohashi2} on the $\mathbb{F}_{p^2}$-maximality/minimality of non-hyperelliptic Howe curves of genus $3$):}

\begin{corollary}[cf.\ {\cite[Theorem 1.1]{Ohashi2}}]\label{cor:max2}
\textcolor{black}{Let $C$ be a non-hyperelliptic Howe curve $C$ of genus $3$ associated with genus-$1$ curves $E_1$ and $E_2$ given by \eqref{eq:Oort0}.}
If $C$ is superspecial, then it is defined over $\mathbb{F}_{p^4}$, and moreover it is \textcolor{black}{$\mathbb{F}_{p^4}$-minimal.}
\end{corollary}

\begin{proof}
Let $E_3$ be a genus-$1$ curve as defined in Theorem \ref{thm:completely decomposed Richelot}, explicitly given by $E_3 : y_3^2 = (x-1)(x-\nu)(x-\mu)(x-\mu \lambda)$ in our case.
If $C$ is superspecial, then $E_1$, $E_2$, and $E_3$ are supersingular, where the Legendre forms $E_1$, $E_2$, and $E_3$ are $E_{\nu} : y^2 = x (x-1)(x-\nu)$, $E_{\lambda} : y^2 = x(x-1)(x-\lambda)$, and $E_{\lambda'} : y^2=x(x-1)(x-\lambda')$ with $\lambda'$ given as in \eqref{eq:lambda_prime}.
It follows from \cite[Proposition 3.1]{AT02} that $\nu$, $\lambda$, and $\lambda'$ are the $4$-th powers of elements in $\mathbb{F}_{p^2}$.
We have
\begin{equation}\label{eq:quad_mu}
\lambda \cdot \mu^2+\frac{\lambda'( \nu+\lambda)-(1+\lambda\nu)}{1-\lambda'}\mu+\nu=0.
\end{equation}
Therefore, it holds that $\mu \in\mathbb{F}_{p^4}$.

\textcolor{black}{
Moreover, it follows also from \cite[Proposition 2.2]{AT02} that $E_{\nu}$, $E_{\lambda}$, and $E_{\lambda'}$ are both $\mathbb{F}_{p^2}$-maximal (resp.\ $\mathbb{F}_{p^2}$-minimal) if $p \equiv 3 \pmod{4}$ (resp.\ $p \equiv 1 \pmod{4}$).
Therefore, $E_1$, $E_2$, and $E_3$ are $\mathbb{F}_{p^4}$-minimal.
Since $C$ is $\mathbb{F}_{p^4}$-minimal if and only if $E_1$, $E_2$, and $E_3$ are both $\mathbb{F}_{p^4}$-minimal, we have proved the second argument.}
\end{proof}

\begin{remark}
In the above \textcolor{black}{corollary}, the curve $C$ is defined over $\mathbb{F}_{p^4}$ in general.
However, from \cite[Theorem 1.1]{Ohashi2}, it can descends to a quartic plane curve defined over $\mathbb{F}_{p^2}$.
\end{remark}

\begin{remark}\label{rem:general}
\textcolor{black}{
The method presented in this subsection might be extended to the case where a generalized Howe curve has genus $\geq 4$ as follows:
Let $C$ be a generalized Howe curve associated with two hyperelliptic curves $C_1$ and $C_2$ of genera $g_1$ and $g_2$ given by \eqref{eq:Howe}, and assume $g_1=g_2$ for simplicity.
We set $\phi:= (x-a_1)(x-a_2)\cdots (x-a_r)$, $Q_1 := (x-b_1)\cdots (x-b_{2g_1+2-r})$, and $Q_2 := (x-c_1)\cdots (x-c_{2g_2+2-r})$.
Putting $Y := y_1/ \phi$ and $Z := y_2/\phi$ and squaring both sides, we obtain $\phi Y^2 = Q_1$ and $\phi Z^2 = Q_2$.
Then, we obtain an equation $f(Y,Z)=0$ by computing $f = \mathrm{Res}_x (f_1,f_2) \in k[Y,Z]$ for $f_1 := Q_1 - \phi Y^2$ and $f_2:= Q_2 - \phi Z^2$.
For the case where $H$ is non-hyperelliptic with $(g_1,g_2,r,g)=(2,2,4,5)$, we prove in \cite{MoriyaKudo2023} that $f$ is absolutely irreducible, and that $C$ is birational to the projective closure of $f=0$.
It is an open problem to prove the irreducibility of $f$ and the birationality for $(g_1,g_2,r,g)$ other than $(1,1,2,3)$ or $(2,2,4,5)$.
}
\end{remark}

\section{Enumeration of superspecial generalized Howe curves}\label{sec:ssp}

This section is devoted to the computational enumeration of superspecial \textcolor{black}{(s.sp.\ for short)} generalized Howe curves.
Note that an algorithm in genus $4$ is already proposed in \cite{KHH} (resp.\ \cite{OKH22}) for the non-hyperelliptic (resp.\ hyperelliptic) case.
We here present computational methods for the case of genus $3$, and prove \textcolor{black}{Theorem \ref{thm:main4}} in Section \ref{sec:intro}.
In our methods, defining equations for produced Howe curves as genus-$3$ curves are explicitly computed by Theorems \ref{thm:Howe curve hyper elliptic} and \ref{thm:Howe curve non-hyper2}, and they are used for classifying isomorphism classes of the curves.

\subsection{Some lemmas for isomorphism classification}\label{subsec:isom}
Before constructing computational methods for enumerating s.sp.\ Howe curves, we start with preparing some lemmas that can make isomorphism classification for Howe curves efficient.

Let $\mathcal{C} = \{ C_1, \ldots , C_n \}$ and $\mathcal{D} = \{ D_1, \ldots , D_n \}$ be sets of curves over \textcolor{black}{$k$}.
We write $\mathcal{C} \sim \mathcal{D}$ if there exists a permutation $\sigma \in \mathfrak{S}_n$ such that $C_i \cong D_{\sigma (i)}$ for each $1 \leq i \leq n$.
The following lemma is a generalization of \cite[Lemma 4]{OKH22}:

\begin{lemma}\label{lem:V4isom}
Let $C^{(1)}$ and $C^{(2)}$ be generalized hyperelliptic Howe curves defined as the normalizations of $C_1^{(1)} \times_{\mathbb{P}^1} C_2^{(1)}$ and $C_1^{(2)} \times_{\mathbb{P}^1} C_2^{(2)}$ respectively, where each $C_{j}^{(i)}$ is a hyperelliptic curve.
Assume that $\mathrm{Aut}(C^{(i)}) = V_4$ for each $i$.
If $C^{(1)} \cong C^{(2)}$, then we have $\{ C_1^{(1)}, C_2^{(1)} \} \sim \{ C_1^{(2)}, C_2^{(2)} \}$.
\end{lemma}

\begin{proof}
\textcolor{black}{
Recall from Subsection \ref{subsec:Howe} that $C_1^{(i)}$ and $C_2^{(i)}$ are not isomorphic over $\mathbb{P}^1$.
Thus, it follows from $\mathrm{Aut}(C^{(i)}) = V_4$ that $\{ C_1^{(i)}, C_2^{(i)} \}$ consists of the quotients of $C^{(i)}$ by non-hyperelliptic involutions, whence the assertion holds by $C^{(1)} \cong C^{(2)}$.}
\end{proof}

In the case where Howe curves are genus-$3$ curves of Oort type, similarly to the proof of Lemma \ref{lem:V4isom}, we can prove the following lemma:

\begin{lemma}\label{lem:V4isomOort}
Let $C^{(1)}$ and $C^{(2)}$ be genus-$3$ Howe curves of Oort type defined as the normalizations of $E_1^{(1)} \times_{\mathbb{P}^1} E_2^{(1)}$ and $E_1^{(2)} \times_{\mathbb{P}^1} E_2^{(2)}$ respectively, where $E_{1}^{(i)}$ and $E_2^{(i)}$ are genus-$1$ curves sharing exactly $2$ ramified points for each $i=1,2$.
Note that $E_{j}^{(i)} \cong C^{(i)} / \langle \sigma_j^{(i)} \rangle$ for some order-$2$ distinct elements $\sigma_{1}^{(i)}, \sigma_{2}^{(i)} \in \mathrm{Aut}(C^{(i)})$ with $\sigma_1^{(i)} \sigma_2^{(i)} = \sigma_2^{(i)} \sigma_1^{(i)}$.
Let $E_3^{(i)}$ be the genus-$1$ curve defined by $C^{(i)} / \langle \sigma_1^{(i)} \sigma_2^{(i)} \rangle$.

We suppose that $\mathrm{Aut}(C^{(i)}) = V_4$ for each $i$.
If $C^{(1)} \cong C^{(2)}$, then one has $\{ E_1^{(1)}, E_2^{(1)}, E_3^{(1)} \} \sim \{ E_1^{(2)}, E_2^{(2)}, E_3^{(2)} \}$.
\end{lemma}

For $\nu, \mu, \lambda \in \overline{\mathbb{F}_p}$ with $\nu \notin \{ 0, 1 \}$, $\mu \notin \{ 0, 1, \nu \}$ and $\lambda \notin \{ 0, 1, \mu^{-1},  \nu \mu^{-1} \}$, we define three genus-$1$ curves as $E_{\nu} : y^2 = x (x-1)(x-\nu)$ and $E_{\mu,\lambda} : y^2 = x (x- \mu ) (x-\mu \lambda)$ and $E_{\nu,\mu,\lambda} : y^2 = (x-1)(x-\nu)(x-\mu)(x-\mu \lambda)$ as in \eqref{eq:OortCurve}.
In the following lemma, we shall bound (by a constant) the number of triples $(\lambda, \mu, \nu)$ such that the three genus-$1$ curves have given $j$-invariants:

\begin{lemma}\label{lem:genus1number}
Given a set $\{ E_1, E_2, E_3 \}$ of three genus-$1$ curves over $\overline{\mathbb{F}_p}$ with $p \neq 2$, the number of triples $(\lambda, \mu, \nu) \in \overline{\mathbb{F}_p}^3$ with $\nu \notin \{ 0, 1 \}$, $\mu \notin \{ 0, 1, \nu \}$, and $\lambda \notin \{ 0, 1, \mu^{-1},  \nu \mu^{-1} \}$ such that $\{ E_1, E_2, E_3 \} \sim \{ E_{\nu}, E_{\mu,\lambda}, E_{\nu,\mu,\lambda} \}$ is at most $2592$ (not depending on $p$!).
\end{lemma}

\begin{proof}
By definition of $\sim$, there exists $\sigma \in \mathfrak{S}_3$ such that $E_{\nu} \cong E_{\sigma(1)}$, $E_{\mu, \lambda} \cong E_{\sigma(2)}$, and $E_{\nu,\mu,\lambda} \cong E_{\sigma(3)}$.
Here we recall a well-known fact that there are at most $6$ values of $\xi$ such that an elliptic curve in Legendre form $y^2 =x (x-1)(x-\xi)$ has a given $j$-invariant.
From this, the number of choices for $\nu$ is at most $6$.
The number of choices for $\lambda$ is also at most $6$, since $E_{\mu,\lambda}$ is isomorphic to $y^2 = x(x-1)(x-\lambda)$.
The third genus-$1$ curve $E_{\nu,\mu,\lambda}$ is isomorphic to $y^2=x(x-1)(x-\lambda')$ for $\lambda'$ as in \eqref{eq:lambda_prime},
and hence $\mu$ can take $6 \times 2 = 12$ values at most.
Since the number of choices for $\sigma$ is $3!=6$, that for $(\nu,\mu,\lambda)$ is bounded by $6^2 \times 12 \times 6 = 2592 $, as desired.
\end{proof}

We also define a genus-$2$ curve $H_{a,b,c}:y^2 = x(x-1)(x-a)(x-b)(x-c)$ (this equation is called a {\it \textcolor{black}{Rosenhain} form}) for mutually distinct elements $a,b,c \in \overline{\mathbb{F}_{p}}$ with $a, b, c \notin \{ 0, 1 \}$.
As for the number of isomorphism classes of $H_{a,b,c}$, we have the following lemma:

\begin{lemma}\label{lem:genus2number}
Given a genus-$2$ curve $H$ over $\overline{\mathbb{F}_p}$ with $p \neq 2$, the number of subsets $\{ a,b,c \} \subset \overline{\mathbb{F}_p}$ of cardinality $3$ with $H_{a,b,c} \cong H$ is at most $120$ (not depending on $p$!).
\end{lemma}

\begin{proof}
\textcolor{black}{
The assertion follows from the proof of {\cite[Lemma 5]{OKH22}} and $\# \mathfrak{S}_3=6$.}
\end{proof}

Note that the bounds given in Lemmas \ref{lem:genus1number} and \ref{lem:genus2number} are best possible.

\subsection{Enumerating s.sp.\ Howe curves of genus three}\label{subsec:enum}
Recall from the beginning of Section \ref{sec:main2} that Howe curves $C$ of genus $3$ are either of {\bf Oort type} or {\bf Howe type}, say $E_1 \times_{\mathbb{P}^1} E_2$ for elliptic curves $E_1$ and $E_2$ as in \eqref{eq:OortCurve}, or $E \times_{\mathbb{P}^1} H$ for curves $E$ and $H$ of genera $1$ and $2$ as in \eqref{eq:HoweType}. 
Note also that any $C$ of Howe type is always hyperelliptic, and recall from Corollary \ref{cor:OortTypeIsHoweType} that any hyperelliptic $C$ of Oort type is of Howe type.
In the following, we present methods to enumerate s.sp.\ Howe curves $C$ of genus $3$, dividing into the hyperelliptic or non-hyperelliptic case:

\subsubsection*{\bf Hyperelliptic case:}
In this case, recall from Lemma \ref{lem:V4} that a Howe curve $C$ is nothing but a hyperelliptic curve with automorphism group containing $V_4$, and it is always of Howe type, namely $C $ is the normalization of $E \times_{\mathbb{P}^1} H$ for curves $E$ and $H$ of genera $1$ and $2$ sharing exactly $4$ Weierstrass points.
Transforming $3$ among the Weierstrass points of $H$ into $\{ 0, 1, \infty \}$, we may assume that $E$ and $H$ are given by $E_{a,b,c}:y^2 = (x-1)(x-a)(x-b)(x-c)$ and $H_{a,b,c}:y^2 = x(x-1)(x-a)(x-b)(x-c)$, and by Lemma \ref{lem:C/sigma hyper} that $C$ is given by $C_{a,b,c} : y^2 = (x^2-1)(x^2-a)(x^2-b)(x^2-c)$. 
Here is our method to enumerate s.sp.\ $C_{a,b,c}$'s.

\begin{method}\label{method1}
For an input $p \geq 3$, proceed with the following:
\begin{enumerate}
    \item[0.] Set $\mathcal{C}$ $\leftarrow$ $\emptyset$ and $\mathcal{I}$ $\leftarrow$ $\emptyset$.
    \item[1.] Compute the list $\mathcal{S}_p$ of all supersingular $j$-invariants in characteristic $p>0$.
    \item[2.] Enumerate all s.sp.\ curves of genus $2$ in characteristic $p$, by computing Richelot isogenies (see \cite[\S 5A]{KHH} or \cite[\S 5.1, Step 1]{OKH22} for details).
    \item[3.] List genus-$2$ s.sp.\ curves of the form $H_{a,b,c}$ by the method in \cite[\S 5.1, Step 2]{OKH22}.
    Every when $H_{a,b,c}$ is produced, check the following:
    \begin{itemize}
        \item The $j$-invariant of an elliptic curve isomorphic to $E_{a,b,c}$ belongs to $\mathcal{S}_p$.
        \item The sequence $I$ of the Shioda invariants of the genus-$3$ hyperelliptic curve $C_{a,b,c}$ does not belong to $\mathcal{I}$.
    \end{itemize}
    If both the two conditions hold, then set $\mathcal{C}$ $\leftarrow$ $\mathcal{C} \cup \{ C_{a,b,c} \}$ and $\mathcal{I}$ $\leftarrow$ $\mathcal{I} \cup \{ I \}$.
    \item[4.] Return $\mathcal{C}$.
\end{enumerate}
\end{method}

\begin{lemma}
Method \ref{method1} outputs a list of all isomorphism classes of s.sp.\ hyperelliptic Howe curves of genus $3$, and its complexity is $\tilde{O}(p^3)$.
\end{lemma}

\begin{proof}
It suffices to show that each computed $C_{a,b,c}$ is superspecial.
Indeed, the $a$-number of each computed $C_{a,b,c}$ associated to $(E_{a,b,c},H_{a,b,c})$ is three by an isomorphism $J(C_{a,b,c})[p] \cong J(E_{a,b,c})[p] \times J(H_{a,b,c})[p]$ of $p$-kernels.

As for the complexity, the complexity of Step 1 is $\tilde{O}(p)$, see, e.g., \cite[\S 5.1.1]{KH20}, and Step 2 is done in $O(p^3)$, see \cite[\S 5.1, Step 1]{OKH22} for details.
In Step 3, we apply the method in \cite[\S 5.1, Step 2]{OKH22}, where at most $120$ (by Lemma \ref{lem:genus2number}) $H_{a,b,c}$'s are produced in constant time for each isomorphism class of s.sp.\ genus-$2$ curves, and the number of produced\ $H_{a,b,c}$'s is $O(p^3)$, namely $\# \mathcal{C} = \mathcal{I} = O(p^3)$.
For each produced $H_{a,b,c}$, it follows from $\# \mathcal{S}_p = O(p)$ that the supersingularity test for $E_{a,b,c}$ is done in $O(\mathrm{log}(p))$, and $I$ is computed in constant time with respect to $p$.
Since $\# \mathcal{I}=O(p^3)$, the cost of testing $I \in \mathcal{I}$ is $\log(p^3) = O(\mathrm{log}(p))$.
Therefore, the total complexity is $p + p^3 + p^3 (\mathrm{log}(p) + \mathrm{log}(p^3)) = \tilde{O}(p^3)$, as desired.
\end{proof}

\begin{remark}\label{rem:method1}
Note that $\# \mathcal{C}$ for an output $\mathcal{C}$ of {\bf Method \ref{method1}} is in fact expected to be $O(p^2)$ as follows:
Heuristically, we can estimate it as the number of s.sp.\ $H_{a,b,c}$ generated in Step 3, times the probability that $E_{a,b,c}$ is supersingular, say
\[
\approx \frac{p^3}{2880} \times 120 \times \frac{1}{2p} = \frac{p^2}{48} = O(p^2).
\]
\end{remark}

S.sp.\ hyperelliptic Howe curves of Oort type can be also enumerated by {\bf Method \ref{method2}} below.
Recall that an Oort type-curve $C$ is given as the normalization of $E_1 \times_{\mathbb{P}^1} E_2$ for $E_1 : y^2 = x (x-1)(x-\nu)$ and $E_2 : y^2 = x (x- \mu ) (x-\mu \lambda)$ together with $E_3 : y^2 = (x-1)(x-\nu)(x-\mu)(x-\mu \lambda)$, and that an equation defining $C$ is \textcolor{black}{\eqref{eq:Howe hyper equation_Legendre}}.

\begin{method}\label{method2}
For an input $p \geq 3$, proceed with the following:
\begin{enumerate}
    \item[0.] Set $\mathcal{C}$ $\leftarrow$ $\emptyset$ and $\mathcal{I}$ $\leftarrow$ $\emptyset$.
    \item[1.] Compute the list $\mathcal{S}_p$ of all supersingular $j$-invariants in characteristic $p>0$, and generate the list $\mathcal{T}_{p}$ of $t \in \mathbb{F}_{p^2}$ such that $ y^2=x(x-1)(x-t)$ is a supersingular elliptic curve.
    \item[2.] For each $\nu$ and $\lambda$ in $\mathcal{T}_p$, conduct the following:
    \begin{enumerate}
        \item[2a.] Compute the roots $\mu$ of $\mu^2 - \nu \lambda^{-1}$.
        (Note that the roots belong to $\mathbb{F}_{p^2}$ as described in the proof of Corollary \ref{cor:max}.)
        \item[2b.] For each root $\mu$, check the following:
        \begin{itemize}
        \item $\mu \notin \{ 0,1,\nu, \lambda^{-1}, \lambda^{-1} \nu, \}$.
        \item The $j$-invariant of an elliptic curve isomorphic to $E_3$ is in $\mathcal{S}_p$.
        \item The sequence $I$ of the Shioda invariants of the genus-$3$ hyperelliptic curve $C$ (associated with $(E_1,E_2)$) as in \textcolor{black}{\eqref{eq:Howe hyper equation_Legendre}} does not belong to $\mathcal{I}$.
    \end{itemize}
        If all the three conditions hold, then set $\mathcal{C}$ $\leftarrow$ $\mathcal{C} \cup \{ C \}$ and $\mathcal{I}$ $\leftarrow$ $\mathcal{I} \cup \{ I \}$.
    \end{enumerate}
    \item[3.] Return $\mathcal{C}$.
\end{enumerate}
\end{method}

Heuristically, the number of s.sp.\ curves we find in Step 2 should be $(\# \mathcal{T}_p)^2$ ($\approx (p-1)^2/4$), times the number of $\mu$'s for each $(\nu,\lambda)$ (at most $2$), times the probability that $E_3$ is supersingular ($\approx 1/(2p)$), and thus $\approx p/4$ in total.

The complexity of {\bf Method \ref{method2}} is estimated as $\tilde{O}(p^2)$, and we expect $\# \mathcal{C} = O(p)$ similarly to Remark \ref{rem:method1}.

\subsubsection*{\bf Non-hyperelliptic case:}
In this case, every non-hyperelliptic Howe curve $C$ of genus $3$ is always of Oort type, namely $C$ is the normalization of $E_1 \times_{\mathbb{P}^1} E_2$ for elliptic curves $E_1$ and $E_2$ as in \eqref{eq:OortCurve} with $\mu^2 \lambda \neq \nu$.
For the enumeration of s.sp.\ $C$'s, we divide the case into the case of $\mathrm{Aut}(C) = V_4$, and the other case.
In the other case, the order of $\mathrm{Aut}(C)$ is either of $8$, $16$, $24$, $48$, $96$, and $168$ (corresponding to $D_4$, $G_{16}$, $\mathfrak{S}_4$, $G_{48}$, $G_{96}$, and $G_{168}$ respectively), see e.g., \cite[\textcolor{black}{Proposition} 2.5]{Ohashi2}.
More precisely, it suffices to proceed with the following {\bf Method \ref{method3}}:

\begin{method}\label{method3}
For an input $p \geq 3$, proceed with the following:
\begin{enumerate}
    \item[0.] Set $\mathcal{J}$ $\leftarrow$ $\emptyset$ and $\mathcal{L}^{{\rm (others)}}_r$ $\leftarrow$ $\emptyset$ for each $r \in \mathcal{R}:= \{ 8,16,24,48,96,168 \} $.
    \item[1.] Compute the list $\mathcal{S}_p$ of all supersingular $j$-invariants in characteristic $p>0$, and for each $j \in \mathcal{S}_{p}$, generate the list $\mathcal{T}_{p,j}$ of at most six $t \in \mathbb{F}_{p^2}$ such that $ y^2=x(x-1)(x-t)$ is a supersingular elliptic curve \textcolor{black}{with $j$-invariant $j$}.
    \item[2.] For each multi-subset $J$ of $\mathcal{S}_p$ with cardinality $3$, conduct the following:
    \begin{enumerate}
        \item[2a.] Set $\mathcal{L}_J$ $\leftarrow$ $\emptyset$.
        \item[2b.] For each permutation $(j_1,j_2,j_3)$ of $J$, and for each $\nu \in \mathcal{T}_{p,j_1}$, $\lambda \in \mathcal{T}_{p,j_2}$, and $\lambda' \in \mathcal{T}_{p,j_3}$, compute $\mu \in \overline{\mathbb{F}_{p}}$ with $\mu \notin \{ 0,1,\nu, \lambda^{-1}, \lambda^{-1} \nu, \pm \sqrt{\lambda^{-1}\nu} \}$ such that the genus-$1$ curve $E_3: y^2=(x-1)(x-\nu)(x-\mu)(x-\mu \lambda)$ is isomorphic to $y^2 = x (x-1)(x-\lambda')$, as the roots of the quadratic equation \eqref{eq:quad_mu}.
        \item[2c.] For each $\mu$ computed in Step 2b, set $\mathcal{L}_J$ $\leftarrow$ $\mathcal{L}_J \cup \{ (\nu,\lambda,\mu) \}$.
    \end{enumerate}
    
    \item[3.] For each multi-subset $J$ of $\mathcal{S}_p$ with cardinality $3$, conduct the following:
    \begin{enumerate}
        \item[3a.] Set $\mathcal{L}^{(V_4)}_J$ $\leftarrow$ $\emptyset$.
        \item[3b.] For each $(\nu,\lambda,\mu) \in \mathcal{L}_J$, \textcolor{black}{compute a sextic given in \eqref{eq:Howe non-hyper equation} and let $C_{\nu,\lambda,\mu}$ be the curve defined by it.
        Compute} $r:=\# \mathrm{Aut}(C_{\nu,\lambda,\mu})$.
        \begin{itemize}
            \item If $r = 4$, then $\mathcal{L}^{(V_4)}_J$ $\leftarrow$ $\mathcal{L}^{(V_4)}_J \cup \{ C_{\nu,\lambda,\mu} \}$.
            \item Otherwise, $\mathcal{L}^{{\rm (others)}}_r$ $\leftarrow$ $\mathcal{L}^{{\rm (others)}}_r \cup \{ C_{\nu,\lambda,\mu} \}$.
        \end{itemize}
        \item[3c.] If $\mathcal{L}^{(V_4)}_J \neq \emptyset$, set $\mathcal{J}$ $\leftarrow$ $\mathcal{J} \cup \{ J \}$.
    \end{enumerate}
    Set $\mathcal{L}^{(V_4)}$ $\leftarrow$ $\bigcup_{J \in \mathcal{J}} \mathcal{L}^{(V_4)}_J$ and $\mathcal{L}^{\rm (others)}$ $\leftarrow$ $\bigcup_{r \in \mathcal{R}} \mathcal{L}^{\rm (others)}_r$.
    \item[4.] Execute the isomorphism classification for each of $\mathcal{L}^{(V_4)}$ and $\mathcal{L}^{\rm (others)}$:
    \begin{enumerate}
        \item[4a.] For each $J \in \mathcal{J}$, initialize $\mathcal{C}^{(V_4)}_J$ $\leftarrow$ $\emptyset$, and conduct the following:
        \begin{itemize}
        \item For each $C_{\nu,\lambda,\mu} \in \mathcal{L}^{(V_4)}_J$, if it is not isomorphic to any curve in $\mathcal{C}^{(V_4)}_J$, we set $\mathcal{C}^{(V_4)}_J$ $\leftarrow$ $\mathcal{C}^{(V_4)}_J \cup \{ C_{\nu,\lambda,\mu} \}$.
        \end{itemize}
        \item[4b.] Similarly to \textcolor{black}{4a}, classify isomorphisms of curves in each $\mathcal{L}^{\rm (others)}_r$, and let $\mathcal{C}^{\rm (others)}_r$ be the list of obtained isomorphism classes.
    \end{enumerate}
    Set $\mathcal{C}^{(V_4)}$ $\leftarrow$ $\bigcup_{J \in \mathcal{J}} \mathcal{C}^{(V_4)}_J$ and $\mathcal{C}^{\rm (others)}$ $\leftarrow$ $\bigcup_{r \in \mathcal{R}} \mathcal{C}^{\rm (others)}_r$.
    \item[5.] Return $\mathcal{C}^{(V_4)} \cup \mathcal{C}^{\rm (others)}$.
\end{enumerate}
\end{method}

Note that the square-root $\sqrt{\lambda^{-1} \nu}$ in Step 2b can be taken as an element in $\mathbb{F}_{p^2}$ by \cite[\textcolor{black}{Proposition} 2.2]{AT02}, and that the roots $\mu$ for each $(\nu,\lambda,\lambda')$ belong to $\mathbb{F}_{p^4}$.

\begin{lemma}\label{lem:method3comp}
Method \ref{method3} outputs a list of all isomorphism classes of s.sp.\ non-hyperelliptic Howe curves of genus $3$, and the arithmetic complexity of each step is given as follows:
Step 1: $\tilde{O}(p)$, Step 2: $O(p^3)$, Step 3: $\tilde{O}(p^3)$, Step 4a: \textcolor{black}{$\tilde{O}(\# \mathcal{J})$}, and Step 4b: $\tilde{O} ((\# \mathcal{L}^{\rm (others)})^2)$, so that the total complexity is upper-bounded by $\tilde{O}(p^3) + \tilde{O}(\# \mathcal{J}) + \tilde{O} ((\# \mathcal{L}^{\rm (others)})^2)$, where the second term is $\tilde{O}(p^3)$. 
\end{lemma}

\begin{proof}
The correctness immediately follows from \textcolor{black}{Theorem \ref{thm:Howe curve non-hyper2} and} Lemma \ref{lem:V4isomOort} together with the following:
Each produced $C$ \textcolor{black}{associated} with $(E_1,E_2,E_3)$ is superspecial since we have an isomorphism $J(C)[p]\cong J(E_1)[p] \times J(E_2)[p] \times J(E_3)[p]$.

Step 1 is done in $\tilde{O}(p)$, and $\# \mathcal{S}_{p}$ follows $O(p)$, which means Step 2 has $O(p^3)$ iterations on $J$.
The total number of iterations for Step 2b is upper-bounded by $6 \times \# \mathcal{T}_{p,j_1} \times \# \mathcal{T}_{p,j_2} \times \# \mathcal{T}_{p,j_3} = 6 \times 6^3$, and solving a quadratic equation to obtain $\mu$ is done in $O(1)$.
Thus, Step 2 has complexity $O(p^3)$.

In Step 3, the number of iterations is just the total number of $(\nu,\lambda,\mu)$ obtained in Step 2, which is upper-bounded by $ (\# \mathcal{S}_p)^3 \times 6^3 \times 2 = O(p^3)$.
For each $(\nu,\lambda,\mu)$, the cost of computing $\# \mathrm{Aut}(C_{\nu,\lambda,\mu})$ is $\tilde{O}(1)$, by using a method described in \textcolor{black}{\cite[Section 3]{LSR21}}.
Therefore, the complexity of Step 3 is $\tilde{O}(p^3)$.

As for Step 4, it follows from the construction of Step 2b that the cardinality of each $\mathcal{L}^{(V_4)}_J$ is at most $\# \mathcal{L}^{(V_4)}_J \leq 6 \times 6^3 \times 2 = 2592$, which is nothing but the number of triples $(\nu,\lambda,\mu)$ as in Lemma \ref{lem:genus1number}.
Therefore, the complexity of Step 4a is upper-bounded by $\# \mathcal{J} \times (\mathrm{max}_{J \in \mathcal{J}} \# \mathcal{L}^{(V_4)}_J)^2 \leq  (\# \mathcal{S}_p)^3 \times 2592^2 =  O(p^3)$, whereas that of Step 4b is
\[
\sum_{r \in \mathcal{R}}(\# \mathcal{L}_r^{\rm (others)})^2 \leq (\# \mathcal{L}^{\rm (others)})^2.
\]
Note that each isomorphism test for two plane quartics is done by computing a Gr\"{o}bner basis in constant time with respect to $p$, say $\tilde{O}(1)$, as in the proof of \textcolor{black}{Theorem \ref{thm:decom richelot non-hyper}}.
\end{proof}

We remark that the number $\# (\mathcal{L}^{(V_4)} \cup \mathcal{L}^{\rm (others)} )$ of $C_{\nu,\lambda,\mu}$'s listed in Steps 2 and 3 is approximately $\left( (p-1)/12 \right)^3 \times 6^3 \times 2 \approx p^3/4$, and thus $\# (\mathcal{C}^{(V_4)} \cup \mathcal{C}^{\rm (others)} ) = O(p^3)$.
In practice, we estimate $\# \mathcal{L}^{\rm (others)}=O(p^2)$ as described in Subsection \ref{subsec:imp} below, and hence the total complexity of Method \ref{method3} is expected to be $\tilde{O}(p^4)$.

\subsection{Implementations and computational results}\label{subsec:imp}

We implemented the three enumeration methods (Methods \ref{method1}, \ref{method2}, and \ref{method3}) described in Subsection \ref{subsec:enum} on Magma V2.26-10~\cite{Magma} in a PC with macOS Monterey 12.0.1, at 2.6 GHz CPU 6 Core (Intel Core i7) and 16GB memory.
We also implemented decision-versions of the three methods in the same environment; a decision-version terminates once a single s.sp.\ curve is found, and outputs ``true'' together with the s.sp.\ curve (otherwise outputs ``false'').
The source codes are available at the following web page:

{\small
\begin{center}
    \texttt{http://sites.google.com/view/m-kudo-official-website/english/code/genus3v4}
\end{center}}

\subsubsection{\bf Enumeration results}
We executed Methods \ref{method1} and \ref{method2} (resp.\ Method \ref{method3}) for every $3 \leq p\leq 200$ (resp.\ $11 \leq p \leq 50$) in the hyperelliptic (resp.\ non-hyperelliptic) case.
Our computational results are summarized in the following theorem:

\begin{theorem}\label{thm:ssp}
For every prime $p$ with $3 \leq p \leq 200$ (resp.\ $11 \leq p \leq 50$), the number of isomorphism classes of s.sp.\ hyperelliptic (resp.\ non-hyperelliptic) Howe curves of genus $3$ over $\overline{\mathbb{F}_{p}}$ is summarized in Table \ref{table:1} (resp.\ Table \ref{table:2}).
\end{theorem}

\subsubsection*{Our observation of computational results and timing data.}
In the hyperelliptic case, we find from Table \ref{table:1} that s.sp.\ Howe curves of type $E \times_{\mathbb{P}^1} H$ exist for every prime $p$ with $p \notin \{3,5,\textcolor{black}{13} \}$, and that the number of isomorphism classes of such s.sp.\ curves is $O(p^2)$ as estimated in Remark \ref{rem:method1}.
On the other hand, there are many $p$ for which a s.sp.\ Howe curve of \textcolor{black}{Oort} type $E_1 \times_{\mathbb{P}^1} E_2$ does not exist, and the number of isomorphism classes is $O(p)$ as we expect at the paragraph just after {\bf Method \ref{method2}}. 
As for timing data, we see that our implementation of {\bf Method \ref{method1}} (resp.\ {\bf Method \ref{method2}}) behaves as $\tilde{O}(p^3)$ (resp.\ $\tilde{O}(p^2)$) estimated in Subsection \ref{subsec:enum}.

\renewcommand{\arraystretch}{0.95}
\begin{table}[H]
\centering{
\caption{The number of isomorphism classes of s.sp.\ {\it hyperelliptic} Howe curves of genus $3$, and timing data for Methods \ref{method1} and \ref{method2} for each $p$ with $3 \leq p \leq 200$.
Note that any hyperelliptic Howe curve of Oort type $E_1 \times_{\mathbb{P}^1} E_2$ is of Howe type $E \times_{\mathbb{P}^1} H$.\\
}
\label{table:1}
\scalebox{0.85}{
\begin{tabular}{c|c||r|r||r|r} \hline
\multirow{2}{*}{$p$} & \multirow{2}{*}{$p \bmod 4$} & Howe type & {Time (s) for \ } & Oort type~ & Time (s) for \ \ \\ 
& & $E \times_{\mathbb{P}^1} H$ & {\bf Method \ref{method1}} &  $E_1 \times_{\mathbb{P}^1} E_2$ & {\bf Method \ref{method2}} \\ \hline
 $3$  & 3 &    0 & $<0.1$ &   0 & $<0.1$ \\ \hline
 $5$  & 1 &    0 & $<0.1$ &   0 & $<0.1$ \\ \hline
 $7$  & 3 &    1 &    1.1 &   1 &    0.2 \\ \hline
 $11$ & 3 &    1 & $<0.1$ &   0 & $<0.1$ \\ \hline
 $13$ & 1 &    0 & $<0.1$ &   0 & $<0.1$ \\ \hline
 $17$ & 1 &    2 & $<0.1$ &   1 & $<0.1$ \\ \hline
 $19$ & 3 &    3 &    0.1 &   0 & $<0.1$ \\ \hline
 $23$ & 3 &    6 &    0.2 &   3 & $<0.1$ \\ \hline
 $29$ & 1 &    3 &    0.2 &   0 & $<0.1$ \\ \hline
 $31$ & 3 &   14 &    0.4 &   5 & $<0.1$ \\ \hline
 $37$ & 1 &   10 &    0.4 &   0 & $<0.1$ \\ \hline
 $41$ & 1 &   10 &    0.5 &   1 & $<0.1$ \\ \hline
 $43$ & 3 &   22 &    0.7 &   0 & $<0.1$ \\ \hline
 $47$ & 3 &   32 &    0.9 &   8 &    0.1 \\ \hline
 $53$ & 1 &   15 &    0.9 &   0 & $<0.1$ \\ \hline
 $59$ & 3 &   45 &    1.4 &   0 & $<0.1$ \\ \hline
 $61$ & 1 &   27 &    1.3 &   0 & $<0.1$ \\ \hline
 $67$ & 3 &   48 &    1.8 &   0 & $<0.1$ \\ \hline
 $71$ & 3 &   87 &    2.4 &  11 &    0.2 \\ \hline
 $73$ & 1 &   53 &    2.4 &   5 &    0.2 \\ \hline
 $79$ & 3 &  106 &    3.2 &  11 &    0.2 \\ \hline
 $83$ & 3 &  100 &    3.3 &   0 &    0.1 \\ \hline
 $89$ & 1 &   89 &    3.9 &   1 &    0.2 \\ \hline
 $97$ & 3 &   55 &    4.8 &   5 &    0.3 \\ \hline
$101$ & 1 &  112 &    5.6 &   0 &    0.2 \\ \hline
$103$ & 3 &  132 &    6.6 &  10 &    0.3 \\ \hline
$107$ & 3 &  166 &    6.9 &   0 &    0.3 \\ \hline
$109$ & 1 &  144 &    7.1 &   0 &    0.3 \\ \hline
$113$ & 1 &  104 &    7.8 &   6 &    0.3 \\ \hline
$127$ & 3 &  182 &   11.2 &   8 &    0.4 \\ \hline
$131$ & 3 &  253 &   13.3 &   0 &    0.4 \\ \hline
$137$ & 1 &  164 &   13.9 &   7 &    0.5 \\ \hline
$139$ & 3 &  234 &   15.3 &   0 &    0.5 \\ \hline
$149$ & 1 &  240 &   18.6 &   0 &    0.5 \\ \hline
$151$ & 3 &  341 &   20.3 &  13 &    0.6 \\ \hline
$157$ & 1 &  194 &   21.6 &   0 &    0.5 \\ \hline
$163$ & 3 &  323 &   24.9 &   0 &    0.6 \\ \hline
$167$ & 3 &  459 &   32.5 &  29 &    0.9 \\ \hline
$173$ & 1 &  232 &   29.5 &   0 &    0.7 \\ \hline
$179$ & 3 &  419 &   32.5 &   0 &    0.7 \\ \hline
$181$ & 1 &  273 &   33.4 &   0 &    0.8 \\ \hline
$191$ & 3 &  629 &   43.0 &  50 &    1.3 \\ \hline
$193$ & 1 & 360 &   42.1 &  17 &    1.0 \\ \hline
$197$ & 1 & 402 &   45.5 &   0 &    0.9 \\ \hline
$199$ & 3 & 547 &   48.1 &  23 &    1.1 \\ \hline
\end{tabular}
 }
}
\end{table}
\renewcommand{\arraystretch}{1}

In the non-hyperelliptic case, Table \ref{table:5} below implies that the number of isomorphism classes of s.sp.\ Howe curves of genus $3$ enumerated by {\bf Method \ref{method3}} is $O(p^3)$ as estimated at the paragraph just after the proof of Lemma \ref{lem:method3comp}. 
We also observe the following from Tables \ref{table:2} -- \ref{table:3}:
\begin{itemize}
\item For each prime $p$ with $13 \leq p \leq 47$, the most expensive step is Step 4.
\item In each case, we measured time for Steps 1 and 2 respectively.
Time for Step 1 is within $0.1$ seconds, and it is negligible compared to Step 2.
We see from Table \ref{table:2} that Step 2 behaves as ${O}(p^3)$ estimated in Lemma \ref{lem:method3comp}.
Table \ref{table:4} implies that the total number $\# (\mathcal{L}^{(V_4)} \cup \mathcal{L}^{\rm (others)} )$ of triples $(\nu,\mu,\lambda)$ obtained in Step 2 such that $C_{\nu,\lambda,\mu}$ is a s.sp.\ non-hyperelliptic Howe curve of genus $3$ follows $O(p^3)$, as expected in Subsection \ref{subsec:enum}.
\item We observe from Table \ref{table:4} that the \textcolor{black}{execution} time for Step 3 \textcolor{black}{behaves as $\tilde{O}(p^3)$}.
As results of Step 3, the number of $(\nu,\lambda,\mu)$ with $\mathrm{Aut}(C_{\nu,\lambda,\mu}) = V_4$ (resp.\ $\mathrm{Aut}(C_{\nu,\lambda,\mu}) \supsetneq V_4$) is $O(p^3)$ (resp.\ $O(p^2)$).
\item As Lemma \ref{lem:V4isomOort} shows, the cardinality $\# \mathcal{L}^{(V_4)}_J$ is bounded by $2592$.
\item As described in Lemma \ref{lem:method3comp} and its proof, we see from Tables \ref{table:2} and \ref{table:3} that Step 4b is costly compared to Step 4a for all $p$ in our computation.
In particular, Table \ref{table:3} also shows that both $\# \mathcal{J}$ and $\# \mathcal{L}^{\rm (others)}$ would follow $O(p^3)$ and $O(p^2)$, and that the execution time of Step 4a (resp.\ 4b) behaves \textcolor{black}{as} $\tilde{O} (\# \mathcal{J} \times (\mathrm{max}_{J \in \mathcal{J}} \# \mathcal{L}^{(V_4)}_J)^2) = \tilde{O}(p^3)$ (resp.\ $\tilde{O}((\# \mathcal{L}^{\rm (others)})^2) = \tilde{O}(p^4)$).
\end{itemize}

\vspace{-3mm}

\renewcommand{\arraystretch}{0.95}
\begin{table}[H]
\centering{
\caption{The total number of isomorphsim classes of s.sp.\ {\it non-hyperelliptic} Howe curves of genus $3$ over $\overline{\mathbb{F}_{p}}$ for small $p$, and their classification by types of automorphism groups.
}
\label{table:5}
\vspace{3mm}
\begin{tabular}{c||r||r||r|r|r|r|r|r||r} \hline
\multirow{2}{*}{$p$} & \multirow{2}{*}{Total} & \multirow{2}{*}{$V_4$~} & \multicolumn{7}{|c}{Others} \\ \cline{4-10} &  &  & $D_4$ \ \ & $G_{16}$ & $S_4$ \ & $G_{48}$ & $G_{96}$ & $G_{168}$  & Total \\ \hline
$11$ &   {\bf 8} &   {\bf 1} &   1 & 0 &  4 & 1 & 1 & 0 &   {\bf 7}  \\ \hline
$13$ &  {\bf 11} &   {\bf 2} &   4 & 0 &  4 & 0 & 0 & 1 &   {\bf 9}  \\ \hline
$17$ &  {\bf 23} &   {\bf 8} &   8 & 0 &  6 & 0 & 0 & 1 &  {\bf 15}  \\ \hline
$19$ &  {\bf 34} &  {\bf 14} &  11 & 1 &  6 & 0 & 1 & 1 &  {\bf 20}  \\ \hline
$23$ &  {\bf 58} &  {\bf 28} &  17 & 1 & 10 & 1 & 1 & 0 &  {\bf 30}  \\ \hline
$29$ & {\bf 119} &  {\bf 70} &  35 & 0 & 14 & 0 & 0 & 0 &  {\bf 49}  \\ \hline
$31$ & {\bf 142} &  {\bf 88} &  38 & 2 & 12 & 0 & 1 & 1 &  {\bf 54}  \\ \hline
$37$ & {\bf 249} & {\bf 168} &  63 & 0 & 18 & 0 & 0 & 0 &  {\bf 81}  \\ \hline
$41$ & {\bf 339} & {\bf 240} &  80 & 0 & 18 & 0 & 0 & 1 &  {\bf 99}  \\ \hline
$43$ & {\bf 394} & {\bf 285} &  85 & 3 & 20 & 0 & 1 & 0 & {\bf 109}  \\ \hline
$47$ & {\bf 509} & {\bf 381} & 103 & 3 & 19 & 1 & 1 & 1 & {\bf 128}  \\ \hline
\end{tabular}
}
\end{table}
\renewcommand{\arraystretch}{1}

\vspace{-5mm}

\renewcommand{\arraystretch}{0.95}
\begin{table}[H]
\centering{
\caption{Timing data for Method \ref{method3}.
For each $p$, the total time is shown in bold fonts.
}
\label{table:2}
\vspace{3mm}
\begin{tabular}{c||r||r|r|r|r} \hline
\multirow{2}{*}{$p$} & Total \ \ &  \multicolumn{4}{|c}{Time (s) for each step}  \\ \cline{3-6}
  & time (s) & 1 \& 2 \ & 3 \ \ \ &  4a \ \ \ & 4b \ \ \  \\ \hline
$11$ & {\bf 59.6}    & $<$ 0.1 &   30.3 &     0.7 &    28.6 \\ \hline
$13$ & {\bf 128.8}   &     0.1 &   64.0 &     4.2 &    60.5 \\ \hline
$17$ & {\bf 158.6}   &     0.1 &   71.8 &    10.4 &    76.2 \\ \hline
$19$ & {\bf 206.1}   &     0.2 &   88.1 &    23.6 &    94.2 \\ \hline
$23$ & {\bf 392.7}   &     0.3 &  171.1 &    53.5 &   167.8 \\ \hline
$29$ & {\bf 526.8}   &     0.6 &  128.3 &   121.8 &   276.1 \\ \hline
$31$ & {\bf 865.5}   &     0.6 &  271.1 &   178.7 &   415.0 \\ \hline
$37$ & {\bf 4289.6}  &     1.8 & 1158.3 &  1125.9 &  2003.6 \\ \hline
$41$ & {\bf 12090.0} &     2.4 &  864.5 &  1339.3 &  9883.8 \\ \hline
$43$ & {\bf 54463.7} &     3.0 &  667.3 & 11732.3 & 42061.2 \\ \hline
$47$ & {\bf 16166.8} &     4.7 & 1082.7 &  1844.0 & 13235.3 \\ \hline
\end{tabular}
}
\end{table}
\renewcommand{\arraystretch}{1}

\renewcommand{\arraystretch}{1}
\begin{table}[t]
\centering{
\caption{The total number of $(\nu, \lambda, \mu) \in \overline{\mathbb{F}_{p}}$ computed in Step 2 of Method \ref{method3} such that $C_{\nu,\lambda,\mu}$ is a s.sp.\ curve over $\overline{\mathbb{F}_{p}}$, and their classification obtained in Step 3 by types of $\mathrm{Aut}(C_{\nu, \lambda, \mu})$.
}
\label{table:4}
\vspace{-2mm}
\begin{tabular}{c||r||r||r||r|r|r|r|r|r} \hline
\multirow{2}{*}{$p$} & Num.\ of   & Time (s) & \multicolumn{7}{|c}{Classification of $(\nu, \lambda, \mu)$'s by $\mathrm{Aut}(C_{\nu,\lambda,\mu})$} \\ \cline{4-10}
& $(\nu, \lambda, \mu)$'s & for Step 3 & $V_4$ \ \ & $D_4$ \ \ & $G_{16}$ & $S_4$ \ & $G_{48}$ & $G_{96}$ & $G_{168}$   \\ \hline
$11$ &    780 &   {\bf 30.3} &    96 &   192 &   0 &  384 & 24 & 84 &  0 \\ \hline
$13$ &   2592 &   {\bf 64.0} &   576 &  1152 &   0 &  768 &  0 &  0 & 96 \\ \hline
$17$ &   3744 &   {\bf 71.8} &  1152 &  1728 &   0 &  768 &  0 &  0 & 96 \\ \hline
$19$ &   4860 &   {\bf 88.1} &  1728 &  2112 &  72 &  768 &  0 & 84 & 96 \\ \hline
$23$ &   6688 &  {\bf 171.1} &  2688 &  2688 &  72 & 1152 & 24 & 84 &  0 \\ \hline
$29$ &  13632 &  {\bf 128.3} &  6624 &  5472 &   0 & 1536 &  0 &  0 &  0 \\ \hline
$31$ &  15492 &  {\bf 271.1} &  7968 &  5664 & 144 & 1536 &  0 & 84 & 96 \\ \hline
$37$ &  25920 & {\bf 1158.3} & 14688 &  9312 &   0 & 1920 &  0 &  0 &  0 \\ \hline
$41$ &  31968 &  {\bf 864.5} & 19008 & 10944 &   0 & 1920 &  0 &  0 & 96 \\ \hline
$43$ &  35964 &  {\bf 667.3} & 22032 & 11712 & 216 & 1920 &  0 & 84 &  0 \\ \hline
$47$ &  43044 & {\bf 1082.7} & 27456 & 12960 & 216 & 2208 & 24 & 84 & 96 \\ \hline
\end{tabular}
}
\end{table}
\renewcommand{\arraystretch}{1}

\renewcommand{\arraystretch}{0.85}
\begin{table}[t]
\centering{
\caption{Time comparison between Steps 4a and 4b of Method \ref{method3}, where $N_{\rm 4a}$ (resp.\ $N_{\rm 4b}$) denotes the number of isomorphism tests executed in Step 4a (resp.\ 4b).
The notations $\mathcal{J}$, $\mathcal{L}^{(V_4)}_J$, and $\mathcal{L}^{\rm (others)}$ are same as in Method \ref{method3}.
}
\label{table:3}
\vspace{-2mm}
\begin{tabular}{c||r|r|r||r||r|r||r} \hline
\multirow{2}{*}{$p$}  & \multicolumn{4}{|c||}{Step 4a} & \multicolumn{3}{|c}{Step 4b}   \\ \cline{2-8}
&  $\# \mathcal{J}$ & $\# \mathcal{L}^{(V_4)}_J$  & $N_{\rm 4a}$ \ \  & Time (s) & $\# \mathcal{L}^{\rm (others)}$ & $N_{\rm 4b}$ \ \ & Time (s) \\ \hline
$11$ &  1 &   $\leq 96$ &     95 &     0.7 &   684 &   1253 &    28.6 \\ \hline
$13$ &  1 &  $\leq 576$ &    862 &     4.2 &  2016 &   4887 &    60.5 \\ \hline
$17$ &  2 &  $\leq 576$ &   2872 &    10.4 &  2592 &  11601 &    76.2 \\ \hline
$19$ &  3 &  $\leq 864$ &   5746 &    23.6 &  3132 &  16840 &    94.2 \\ \hline
$23$ &  6 &  $\leq 768$ &   8708 &    53.5 &  4020 &  30486 &   167.8 \\ \hline
$29$ &  7 & $\leq 1728$ &  46730 &   121.8 &  7008 & 100367 &   276.1 \\ \hline
$31$ &  9 & $\leq 1728$ &  57752 &   178.7 &  7524 & 126870 &   415.0 \\ \hline
$37$ & 10 & $\leq 2592$ & 172200 &  1125.9 & 11232 & 317871 &  2003.6 \\ \hline
$41$ & 16 & $\leq 2592$ & 202800 &  1339.3 & 12960 & 479037 &  9883.8 \\ \hline
$43$ & 19 & $\leq 2592$ & 241203 & 11732.3 & 13932 & 554135 & 42061.2 \\ \hline
$47$ & 29 & $\leq 2592$ & 272739 &  1844.0 & 15588 & 726460 & 13235.3 \\ \hline
\end{tabular}
}
\end{table}
\renewcommand{\arraystretch}{1}

\subsubsection{\bf Existence results and fields of definition}

For the hyperelliptic case, we have the following results for every prime $p$ with $3 \leq p \leq 200$: 
\begin{enumerate}
 \item We see from Table \ref{table:1} that there is no s.sp.\ hyperelliptic Howe curve of genus $3$ if $p \in \{ 3,5,13 \}$, and otherwise such curves do exist.
 \item[(2)] We also examined with Magma that the obtained s.sp.\ hyperelliptic curves in Table \ref{table:1} are all $\mathbb{F}_{p^2}$-maximal (resp.\ $\mathbb{F}_{p^2}$-minimal) when $p \equiv 3 \pmod{4}$ (resp.\ $p \equiv 1 \pmod{4}$), as Corollary \ref{cor:max} implies.
    In fact, when $p \equiv 3 \pmod{4}$ with $p \neq 3$, we succeeded in finding a triple $(a,b,c) \in (\mathbb{F}_{p^4})^3$ such that $C_{a,b,c} : y^2 = (x^2-1)(x^2-a)(x^2-b)(x^2-c)$ is a s.sp.\ curve of genus $3$ defined over the prime field $\mathbb{F}_{p}$, namely each $x^i$-coefficient in $(x^2-1)(x^2-a)(x^2-b)(x^2-c)$ belongs to $\mathbb{F}_{p}$.
\end{enumerate}
We also investigate the existence of non-hyperelliptic Howe curves of genus $3$ for $p$ larger than ones in Table \ref{table:5}.
For this, we executed the decision-version (described in the first paragraph of Section \ref{subsec:imp}) of Method \ref{method3}, and obtain the following results for every prime $p$ with $11 \leq p \leq 20000$:
\begin{enumerate}
 \item[(3)] For every $p$, there exists a s.sp.\ non-hyperelliptic Howe curve of genus $3$.
 \item[(4)] When $p \equiv 3 \pmod{4}$, we succeeded in finding a triple $(\nu,\lambda,\mu) \in (\mathbb{F}_{p^4})^3$ such that the corresponding non-hyperelliptic Howe curve of genus $3$ is a s.sp.\ curve defined over $\mathbb{F}_{p}$.
 However, in the case where $p \equiv 1 \pmod{4}$, there are many $p$ such that any non-hyperelliptic s.sp.\ Howe curve of genus $3$ constructed as above is not defined over $\mathbb{F}_{p}$.
\end{enumerate}
Note that the execution time for each $p$ is at most a few seconds (it takes a few hours in total to terminate for all $p$).

The existence result (1) (resp.\ (3)) examines Oort's result~\cite[Theorem 5.12 (3)]{oort1991hyperelliptic} in the hyperelliptic case with $p \equiv 3 \pmod{4}$ (resp.\ \cite[Theorem 5.12 (1)]{oort1991hyperelliptic} in the non-hyperelliptic case).

As for the maximality/minimality and the field of definition, Ibukiyama showed in \cite{Ibukiyama} that there exists a s.sp.\ curve $C$ of genus $3$ defined over the {\it prime} field $\mathbb{F}_{p}$ which is $\mathbb{F}_{p^2}$-maximal for each odd prime $p$, but it is not known whether each $C$ is hyperelliptic or non-hyperelliptic.
As such a curve $C$, we expect from our computational results (2) and (4) together with Corollary \ref{cor:max} that we can take it to be \textcolor{black}{either a hyperelliptic curve or a non-hyperelliptic Howe curve} when $p \equiv 3 \pmod{4}$, but for the case $p \equiv 1 \pmod{4}$, we need to consider a curve that is not realized as a Howe curve.

\subsection{Some remarks toward higher genus}
Taking $C_1$ and $C_2$ in Subsection \ref{subsec:Howe} to be curves of various genera, construction and enumeration of s.sp.\ Howe curves in higher genus would be also possible, see \cite{Howe}, \cite{KHH}, and \cite{OKH22} for the genus-$4$ case, and \cite{Howe2017} for non-hyperelliptic case of genera $5$, $6$, and $7$. 
We here focus on the case where the resulting Howe curve is hyperelliptic, and restrict ourselves to the case of genera $4$, $5$, and $6$ for simplicity.
Let $C_1$ and $C_2$ be hyperelliptic curves of genera $g_1$ and $g_2$ over \textcolor{black}{$k$}, and $C$ the desingularization of $C_1 \times_{\mathbb{P}^1} C_2$. 
To construct s.sp.\ hyperelliptic Howe curves in genera $4$, $5$, and $6$, it suffices from Lemma \ref{lem:V4} that we take the following values as $g_1$ and $g_2$:
\begin{enumerate}
    \item[{\bf (I)}] $g_1 = g_2 = 2$: In this case, $g = 4$.
    \item[{\bf (II)}] $g_1 =2$, $g_2 = 3$: In this case, $g= 5$.
    \item[{\bf (III)}] $g_1 = g_2 = 3$: In this case, $g= 6$.
\end{enumerate}
For {\bf (I)}, Ohashi-Kudo-Harashita~\cite{OKH22} recently proposed an enumeration algorithm with complexity $O(p^3)$ restricting to the case where $\mathrm{Aut}(C) = V_4$, and succeeded in finding s.sp.\ curves in every characteristic $p$ with $19 \leq p \leq 6691$.
The reason why we can efficiently enumerate s.sp.\ Howe curves in this case is that we already have an efficient algorithm~\cite[\S 5A]{KHH}, which enumerates all genus-$2$ s.sp.\ curves in $O(p^3)$ by producing curves whose Jacobian varieties are Richelot isogenous to those of some given $\mathbb{F}_{p^2}$-maximal curves of genus $2$.

Regarding {\bf (II)} and {\bf (III)}, the enumeration is of course possible, but the complexity would becomes exponential due to the following reason:
In the case of genus $3$, there is no polynomial-time algorithm as in \cite[\S 5A]{KHH} at this point. 
A method with exponential complexity for the genus-$3$ case is the following:

\begin{itemize}
\item It requires to search among genus-$3$ curves of the form
\begin{equation}\label{eq:genus3}
    D : y^2 = x(x-1)(x^5+a_4 x^4 + a_3 x^3 + a_2 x^2 + a_1 x + a_0)
\end{equation}
with $a_i \in \mathbb{F}_{p^2}$, taking the case $\mathrm{Aut}(D) \not\supset V_4$ into account, and it suffices to solve a zero-dimensional multivariate system in $a_i$ obtained from the condition ``the Cartier-Manin matrix of $D$ is zero'' by the Gr\"{o}bner basis computation as in \cite{KH18ful}, \cite{KH18}.
However, the complexity of this method would be exponential with respect to $p$ in \textcolor{black}{the} worst case.

\end{itemize}
If we restrict ourselves to the case of $\mathrm{Aut}(D) \supset V_4$, we can apply the methods in Subsection \ref{subsec:enum}, and the total complexity \textcolor{black}{would be bounded by a polynomial function of $p$}.

Once s.sp.\ genus-$3$ hyperelliptic curves are obtained, we can construct s.sp.\ curves of genus $5$ and $6$, as in the method in \cite{OKH22} (we will study details in \cite{Kudo2022}).

\subsection*{Acknowledgments}
\textcolor{black}{The authors thank the anonymous referee for his/her comments and suggestions, which have helped the authors significantly improve the paper.}
The authors also thank Shushi Harashita, Tomoyoshi Ibukiyama, Toshiyuki Katsura, Ryo Ohashi, and Katsuyuki Takashima for helpful comments.
This work was supported by JSPS Research Fellowship for \textcolor{black}{Young} Scientists 21J10711, and JSPS Grant-in-Aid for Young Scientists 20K14301.

\bibliography{richelot}
\bibliographystyle{plain}

\appendix

\section{Some properties of a certain bivariate quartic}

\textcolor{black}{
Let $f(Y,Z)$ be a bivariate quartic polynomial of the form 
\begin{equation}\label{eq:f}
    f=b_{40}Y^4 + b_{22} Y^2Z^2  + b_{20} Y^2  + b_{04}Z^4 + b_{02} Z^2+ b_{00}
\end{equation}
in the variables $Y$ and $Z$ over $k$.
Assume $b_{40},b_{04},b_{00} \neq 0$.}

\begin{lemma}\label{lem:linear}
\textcolor{black}{
With notation as above, assume that $f$ has a linear factor.
Then we have the following system of equations:
\begin{equation}\label{eq:system1}
    b_{22}^2 - 4 b_{40} b_{04} = b_{20}^2 - 4 b_{40}b_{00} = b_{22} b_{20} + 2 b_{40} b_{02} = 0.
\end{equation}}
\vspace{-4mm}
\end{lemma}

\begin{proof}
\textcolor{black}{
Assume that $f$ has a linear factor $L_1 \in k[Y,Z]$.
From our assumptions $b_{40},b_{04},b_{00} \neq 0$, we may suppose $L_1 = Y + a_1 Z + a_2$, where $a_1, a_2 \in k \smallsetminus \{ 0 \}$.
Then, $f$ is also divided by
\[
\begin{aligned}
    L_2 &:= - L_1(-Y,Z) = Y - a_1 Z - a_2,\\
    L_3 &:= L_1(Y,-Z) = Y - a_1 Z + a_2,\\
    L_4 &:= - L_1(-Y,-Z) = Y + a_1 Z - a_2.\\
\end{aligned}
\]
If $L_1$, $L_2$, $L_3$, and $L_4$ are all different, then $f =b_{40}L_1 L_2 L_3 L_4$, so that
\begin{eqnarray}
f &=& b_{40}(Y^2 - (a_1Z+a_2)^2 )(Y^2 - (a_1Z-a_2)^2) \nonumber \\
&=& b_{40}(Y^4 -2(a_1^2Z^2+a_2^2) Y^2+(a_1^2Z^2 - a_2^2)^2) \nonumber \\
&=& b_{40}(Y^4 - 2a_1^2 Y^2 Z^2 - 2 a_2^2 Y^2 + a_1^4 Z^4 - 2 a_1^2 a_2^2 Z^2 + a_2^4).\label{eq:f2}
\end{eqnarray}
Comparing the coefficients in \eqref{eq:f} and \eqref{eq:f2}, we obtain the system \eqref{eq:system1}.}

\textcolor{black}{
We consider the case where two of $L_1$, $L_2$, $L_3$, and $L_4$ are equal to each other.
Without loss of generality, we may suppose that $L_1$ is equal to one of the others.
In this case, it follows from $a_1 \neq 0$ that $L_1 \neq L_2$ and $L_1 \neq L_3$, whence $L_1 = L_4$.
Therefore, we have $a_2 = 0$, so that $f$ is divisible by $Y^2 - a_1^2 Z^2$.
A straightforward computation shows
\[
\begin{aligned}
    f =& (Y^2 - a_1^2 Z^2) Q + (b_{40} a_1^4 + b_{22}a_1^2 + b_{04})Z^4 + (b_{20}a_1^2 + b_{02})Z^2 + b_{00}
\end{aligned}
\]
with $Q= b_{40} Y^2 + (b_{40}a_1^2 + b_{22}) Z^2  + b_{20}$, a contradiction to $b_{00} \neq 0$.}
\end{proof}



\begin{lemma}\label{lem:Q}
\textcolor{black}{
With notation as above, if $f$ is factored into $f = b_{40} Q_1 Q_2$ for some monic irreducible quadratic polynomials $Q_1$ and $Q_2$, then either of the following four cases holds:
\[
\begin{aligned}
   & {\bf (I)} \
    b_{22}^2 - 4 b_{40} b_{04} = 0.
    \quad
    {\bf (II)} \
    \begin{cases}
       b_{22}^2 - 4 b_{40} b_{04} = 0,\\
       b_{20}^2 - 4 b_{40} b_{00} = 0.
    \end{cases}
    \quad
     {\bf (III)} \
    \begin{cases}
       b_{20}^2 - 4 b_{40} b_{00} = 0,\\
       b_{02}^2 - 4 b_{04}b_{00} = 0.
    \end{cases}\\
    & {\bf (IV)} \ 2b_{40} b_{02} - b_{22} b_{20} \pm \sqrt{b_{22}^2 - 4 b_{40} b_{04}} \sqrt{b_{20}^2 - 4 b_{40} b_{00}} = 0.
\end{aligned}
\]}
\end{lemma}

\begin{proof}
\textcolor{black}{
A tedious computation implies that there are four cases on $(Q_1,Q_2)$:
 \[
 \begin{aligned}
    & {\bf (I)} \
    \begin{cases}
       Q_1 = Y^2 + a_2 Z^2 + a_3 Y + a_5,\\
       Q_2 = Y^2 + a_2 Z^2 - a_3 Y + a_5.
    \end{cases}
    \quad
    {\bf (II)} \
    \begin{cases}
       Q_1 = Y^2 + a_2 Z^2 + a_4 Z + a_5,\\
       Q_2 = Y^2 + a_2 Z^2 - a_4 Z + a_5.
    \end{cases}\\
    & {\bf (III)} \
     \begin{cases}
       Q_1 = Y^2 + a_1 Y Z + a_2 Z^2 + a_5,\\
       Q_2 = Y^2 - a_1 Y Z + a_2 Z^2 + a_5.
    \end{cases}
    \quad
    {\bf (IV)} \
    \begin{cases}
        Q_1 = Y^2 + a_2 Z^2 + a_5,\\
       Q_2 = Y^2  + a_2' Z^2 + a_5'.
    \end{cases}
    \end{aligned}
    \]
    In each case, $f$ is expanded as follows:
    \[
    \begin{array}{cl}
       {\bf (I)} & f = b_{40} (Y^4 + 2a_2 Y^2 Z^2 + (-a_3^2 + 2 a_5) Y^2 + a_2^2 Z^4 +2 a_2 a_5 Z^2 + a_5^2),\\
       {\bf (II)} & f = b_{40} (Y^4 + 2a_2 Y^2 Z^2 + 2 a_5 Y^2 + a_2^2 Z^4 +(2 a_2 a_5 - a_4^2) Z^2 + a_5^2),\\
       {\bf (III)} & f = b_{40} (Y^4 + (-a_1^2 + 2 a_2) Y^2 Z^2 + 2 a_5 Y^2 + a_2^2 Z^4 +2 a_2 a_5  Z^2 + a_5^2),\\
       {\bf (IV)} & f = b_{40} (Y^4 + (a_2 + a_2') Y^2 Z^2 + (a_5 +a_5')Y^2 + a_2 a_2' Z^4 +(a_2 a_5' +a_5 a_2')  Z^2 + a_5 a_5').
    \end{array}
    \]
As for the case (IV), we have the following:
\begin{itemize}
        \item From the coefficients of $Y^2Z^2$ and $Z^4$, the elements $b_{40}a_2$ and $b_{40}a_2'$ are the roots of $X^2 -b_{22} X + b_{40}b_{04}$, whence
        \[
        2b_{40}a_2, 2b_{40}a_2' = {b_{22} \pm \sqrt{b_{22}^2 - 4 b_{40} b_{04}}}.
        \]
        \item From the coefficients of $Y^2$ and $1$, the elements $b_{40}a_5$ and $b_{40}a_5'$ are the roots of $X^2 -b_{20} X + b_{40}b_{00}$, whence
        \[
        2b_{40}a_5, 2b_{40}a_5' = b_{20} \pm \sqrt{b_{20}^2 - 4 b_{40} b_{00}}.
        \]
        \item Once $a_2$, $a_2'$, $a_5$, and $a_5'$ are specified, from the coefficient of $Z^2$, we have $b_{40}(a_2 a_5' + a_5 a_2') = b_{02}$, whence
        \[
        \begin{aligned}
        4b_{40} b_{02} &= 4 b_{40}^2 (a_2 a_5' + a_5 a_2') =(2b_{40}a_2)(2b_{40}a_5') +  (2b_{40}a_2')(2b_{40}a_5) \\
        &= 2 \left(b_{22} b_{20} \pm \sqrt{b_{22}^2 - 4 b_{40} b_{04}} \sqrt{b_{20}^2 - 4 b_{40} b_{00}} \right),
        \end{aligned}
        \]
        as desired.
    \end{itemize}}
\end{proof}

\end{document}